\documentclass[11pt]{article}%
\usepackage{amssymb,amsmath,amsfonts,amsthm,array,bm,bbm,color}%
\usepackage{graphicx}
\usepackage{hyperref}
\hypersetup{colorlinks=true, linkcolor=red, anchorcolor=blue, citecolor=blue}
\usepackage{graphicx}

\usepackage{tikz}

\providecommand{\U}[1]{\protect\rule{.1in}{.1in}}

\numberwithin{equation}{section}

\newtheorem{theorem}{Theorem}[section]
\newtheorem{lemma}[theorem]{Lemma}
\newtheorem{corollary}[theorem]{Corollary}
\newtheorem{proposition}[theorem]{Proposition}
\newtheorem{remark}[theorem]{Remark}

\newtheorem{definition}[theorem]{Definition}
\newtheorem{assumption}[theorem]{Assumption}

\def\<{\langle}
\def\>{\rangle}
\def\d{\,{\rm d}}
\def\&{\,&}

\def\div{{\rm div}}

\def\E{\mathbb{E}}
\def\N{\mathbb{N}}
\def\P{\mathcal{P}}
\def\R{\mathbb{R}}

\def\F{\mathcal{F}}
\def\B{\mathcal{B}}
\def\p{\partial}
\def\Id{\operatorname{Id}}
\def\law{\operatorname{Law}}
\def\eps{\varepsilon}
\def\vphi{\varphi}
\def\M{\mathcal{M}}

\def\W{\mathcal{W}}

\def\t{\widetilde}

\begin{document}

\title{From Kac particles to the Landau equation with hard potentials: BBGKY hierarchy method}

\author{Shuchen Guo\footnote{Email: guo@maths.ox.ac.uk. Mathematical Institute, University of Oxford, Oxford, OX2 6GG, UK.}}

\maketitle

\vspace{-20pt}

\begin{abstract}

We study the Kac particle model for the space-homogenous Landau equation with hard potentials. By showing a sharper Povzner-type inequality, we obtain the uniform-in-time and uniform-in-$N$ propagation of exponential moment for the first marginal of the solution of the many-particle Liouville equation. This key property enables us to show the uniqueness of weak solutions of the corresponding infinite Landau hierarchy by coupling method. As a result, we prove the propagation of chaos for the Landau equation with hard potentials.
\end{abstract}



\section{Introduction}\label{sec intro}  

Lev Landau derived a nonlinear and nonlocal diffusion equation, now known as the Landau equation, from the Boltzmann equation via the \textit{grazing collision limit} \cite{landau1936kinetische}. This equation plays a fundamental role in plasma physics to describe weakly collisional interactions. In this work, we consider the space-homogenous Landau equation with hard potentials in $\R^3$, given by
\begin{equation}\label{eq Landau}
\p_t f=\div\left[(A\ast f)\nabla f-(B\ast f)f \right],
\end{equation}
where the matrix-valued function $A$ and the vector-valued function $B$ are defined by
\begin{equation}\label{def AB}
A(z):=|z|^{\gamma+2}\Pi(z),\quad\text{with}\quad\Pi(z)=\Id-\frac{z\otimes z}{|z|^2},\qquad B(z):=-2z|z|^\gamma, \quad\text{for}\quad \gamma\in(0,1].    \end{equation}
Our main result derives the Landau equation \eqref{eq Landau} from Kac particles under \textit{mean-field} scaling via the \textit{BBGKY hierarchy method}. 

The so-called BBGKY hierarchy (named after Bogoliubov, Born, Green, Kirkwood and Yvon) is a string of equations satisfied by the marginals of an $N$-particle distribution. The BBGKY hierarchy method provides a systematic framework to show the propagation of chaos for interacting particles. In general, the method consists of two main steps: (i) Write the BBGKY hierarchy on marginals of the solution of Liouville equation for $N$-particle systems and show that, as $N \to \infty$, the marginals converge  to a solution of an infinite hierarchy; (ii) Prove that solutions of this infinite hierarchy are unique and coincide with tensor products of the solution of the limiting PDE. In functional analytic framework, these two steps are sometimes referred to as the consistency estimate and the stability estimate, respectively. 

We consider the \textit{unmodified Kac article} systems associated with the Landau equation with hard potentials. The corresponding Liouville equation for the $N$-particle distribution is given by
\begin{equation}\label{eq Liouville}
\begin{aligned}
\p_t f_{N}=\frac{1}{2N}\sum_{i\neq j}\div_{v^i-v^j}\left[A(v^i-v^j)(\nabla_{v^i}-\nabla_{v^j})f_{N}\right],   
\end{aligned}
\end{equation}
with initial data $f_N(0)=(f^0)^{\otimes N}$. 
One can formally associate to the Liouville equation \eqref{eq Liouville} a conservative (Kac-like) $N$-particle stochastic system, given by 
\begin{equation}\label{kacparticles}
\d V_t^i=\frac{2}{N}\sum_{j\neq i}^{N}B(V_t^i-V_t^j)\d t+\sqrt{\frac{2}{N}}\sum_{j\neq i}^{N}\sigma(V_t^i-V_t^j)\d Z_t^{i,j},
\end{equation}
where, for $1\leq i<j\leq N$,  $Z_t^{i,j}=W_t^{i,j}$ are $N(N-1)/2$   i.i.d.  $3$-dimensional Brownian motions, while $Z_t^{j,i}=-W_t^{i,j}$ are anti-symmetric. The diffusion matrix $\sigma$ is a square root of the symmetric positive semi-definite matrix $A$, i.e., $\sigma\sigma^T=A$, defined as $\sigma(z)=|z|^{1+\gamma/2}\Pi(z)$.
Assume that the initial data $(V_0^1,\ldots,V_0^N)$ are i.i.d with the distribution $f^0\in \P(\R^3)$. Under above assumptions, the particles are exchangeable, and the total momentum and kinetic energy are conserved almost surely:
\begin{equation}\label{eq conservation}
\sum_{i=1}^NV_t^i=\sum_{i=1}^NV_0^i,\quad\text{and}\quad\sum_{i=1}^N|V_t^i|^2=\sum_{i=1}^N|V_0^i|^2,\quad\text{a.s.}
\end{equation}
For any $m\in\N$ and $1\leq m\leq N-1$, we define the $m$-marginal $f_{N,m}$ by integrating over the variables from $v^{m+1},\ldots,v^N$,
\begin{equation}\label{def marginal}
f_{N,m}=\int_{\R^{3(N-m)}}f_N\d v^{m+1}\cdots\d v^{N}.
\end{equation}
This marginal describes joint distribution of the first $m$ particles in the Kac particles system. Formally, the marginal satisfies the \textit{BBGKY hierarchy}:
\begin{equation}\label{eq BBGKY}
\begin{aligned}
\p_t f_{N,m}
=\&\frac{1}{2N}\sum_{i\neq j\leq m}\div_{v^i-v^j}\left[A(v^i-v^j)\nabla_{v^i-v^j}f_{N,m}\right]\\\&+\frac{N-m}{N}\sum_{i=1}^m\div_{v^i}\int_{\R^3}\left[A(v^i-v^{m+1})\nabla_{v^i-v^{m+1}}f_{N,m+1}\right]\d v^{m+1},
\end{aligned}
\end{equation}
with initial data $f_{N,m}(0)=(f^0)^{\otimes m}$. We now state the assumptions on the initial data $f^0$; see Remark \ref{rmk initial} for further discussion. 
\begin{assumption}\label{assumption hard kac}
The initial data $f^0$ satisfies the following conditions: 
\begin{itemize}
\item[1.] $f^0$ is a probability measure ($f^0\in\P(\R^3)$) , and its density (also denoted by $f^0$) is supported in the ball centred at the origin and with radius $r_0>0$.
    \item [2.] $f^0$ has finite entropy: $\int_{\R^{3}}f^0\log f^0<+\infty$.
\end{itemize}
\end{assumption}
 
The study of BBGKY hierarchies dates back to the 1940s and has been extensively used to derive the Boltzmann equation under different scaling limits. For the Boltzmann equation for hard spheres, Lanford \cite{lanford2005time} proved the propagation of chaos holds in a short time in the \textit{Boltzmann-Grad limit}. The result was later generalised to short-range potentials by King \cite{king1975bbgky}. A more  refined analysis can be found in \cite{gallagher2013newton}.
In the context of space-homogeneous Boltzmann equations for hard spheres and Maxwellian molecules, Mischler and Mouhot \cite{mischler2013kac} revisited the BBGKY method within a functional analytic framework. The BBGKY hierarchy method has also been extended to quantum systems; for instance, \cite{chen2023derivation} derives the quantum Boltzmann equation under \textit{weak-coupling} scaling.

Beyond the Boltzmann equation, the BBGKY hierarchy method has been successfully applied to derive other nonlinear equations. For instance, the cubic nonlinear Schrödinger equation is derived from quantum many-body systems under a moderately interacting scaling; see \cite{erdos2006derivation, erdos2007derivation}. The same authors also derived the Gross–Pitaevskii equation in the local interaction scaling; see \cite{erdos2010derivation}. The compressible Euler equation is derived under a moderate scaling in the quantum setting; see \cite{chen2024derivation}. In the mean-field scaling, which is the focus of this work, derivations of the Vlasov–Fokker–Planck equation and the McKean-Vlasov equation with singular interactions have been carried out in \cite{bresch2022new}. We also refer the reviews \cite{jabin2014review,golse2016dynamics}, where the BBGKY hierarchy method is discussed in the context of mean-field limits. More recently, it is found out that the BBGKY hierarchy approach has also been applied to obtain the sharp rate for propagation of chaos estimates. This was initiated in \cite{lacker2023hierarchies} by Lacker, and subsequently extended to the model with singular interactions in \cite{han2022entropic} and \cite{wang2024sharp}. 

In \cite{Miot2011kac}, the authors initiated a BBGKY-based approach to derive the Landau equation, showing that the marginals of the \textit{modified} Liouville equation associated with Kac particles converge to the infinite Landau hierarchy with Coulomb potential ($\gamma = -3$ in \eqref{def AB}). Subsequently, the first rigorous compactness result for the \textit{unmodified} Kac particle system associated with the Landau equation for the Coulomb potential was obtained in \cite{carrillo2025fisher}. There, the authors established the well-posedness and compactness of the unmodified Liouville equation by proving the monotonicity of the Fisher information. It is observed in \cite{carrillo2025fisher} that the lifted version of Landau equation used by Guillen and Silvestre in \cite{guillen2025landau} is the Liouville equation \eqref{eq Liouville} with $N=2$.  However, as in \cite{Miot2011kac}, a complete propagation of chaos result was not obtained, due to the lack of an effective stability estimate, which is essential in proving the uniqueness of solutions of the infinite Landau hierarchy, which is formally obtained by sending $N\to\infty$  in \eqref{eq BBGKY} and reads as for any $m\in\N$: 
\begin{equation}\label{eq Landau hierarchy}
\p_t f_{m}=\sum_{i=1}^m\div_{v^i}\left[\int_{\R^3}A(v^i-v^{m+1})\left(\nabla_{v^i}f_{m+1}-\nabla_{v^{m+1}}f_{m+1}\right)\d v^{m+1}\right],
\end{equation}
with initial data $f_m(0)=(f^0)^{\otimes m}$. 

Regarding the uniqueness of infinite hierarchies, Narnhofer and Sewell first established the uniqueness of the Vlasov hierarchy in \cite{narnhofer1981vlasov}. This result was later generalised by Spohn in \cite{spohn1981vlasov} using a functional-analytic approach based on the Hewitt–Savage theorem. For a quantitative stability estimate in this setting, see also \cite{golse13empirical}. Uniqueness results for the Gross–Pitaevskii hierarchy have been obtained through combinatorial arguments in \cite{erdos2007derivation} and \cite{klainerman2008uniqueness}. Also, the board game method introduced in \cite{klainerman2008uniqueness} has since been adapted to establish uniqueness for the Boltzmann hierarchy, as in \cite{chen2019local, ampatzoglou2025global}. More recently, in \cite{bresch2022new}, the authors proved the uniqueness of the Vlasov–Fokker–Planck hierarchy. Their strategy involves first deriving a hierarchy of ordinary differential inequalities and then obtaining a uniform bound via an iterative argument. In our work, we adopt a similar approach by adapting the SDE coupling method to derive a differential inequality hierarchy. Our analysis is more complicated than that in \cite{bresch2022new}, due to the presence of a cut-off on the relative velocity. To address this difficulty, we  show the uniform-in-time propagation of
exponential moments for the associated Kac particles, and follow the strategy introduced in \cite{lacker2023hierarchies} to obtain a refined stability estimate for the hierarchy. 

As mentioned above, a crucial ingredient in our proof is the propagation of exponential moments, which is usually a key argument in kinetic theory. This idea was initiated in \cite{bobylev1997moment} for the space-homogeneous Boltzmann equation for hard spheres. Since then, it has been extended and refined in various settings; see, for example, \cite{alonso2013new,cao2024propagation}. For the Landau equation with hard potentials, the propagation of exponential moments was established in \cite{fournier2017kac}. Furthermore, \cite{fournier2021stability} showed that exponential moments are created for all $t>0$, meaning they become finite after any positive time. See also \cite{fournier2021exponential} and the references therein.

However, obtaining exponential moment estimates for the many-particle evolution, namely, the Liouville equation, can be more challenging. Existing results typically give polynomial moments estimates. Regarding Kac particle systems, the propagation of polynomial moments has been investigated in \cite{sznitman1984equations}; see also  \cite{mischler2013kac,cortez2018quantitative,liu2024rate} for the Boltzmann equation with Maxwellian molecules and hard potentials. For the Landau equation, similar results have been obtained in \cite{carrapatoso2016propagation} for Maxwellian molecules case, and in \cite{fournier2017kac} for hard potentials case. In our manuscript, we fill the gap between polynomial moments and exponential ones for Kac particles associated to Landau equation with hard potentials, where we prove a sharper version of Povzner-type inequality by assuming a stronger initial data, and use an expansion argument to carefully sum up all integer exponents. The moment result reads as follows.
\begin{theorem}[Propagation of the exponential moment]\label{thm exponential}
Under Assumption \ref{assumption hard kac}, there exists some constant $\xi_{r_0,\gamma}$ depending only on $r_0$ and $\gamma$, such that the first marginal of the weak solution of \eqref{eq Liouville} satisfies the following  uniform-in-time and uniform-in-$N$ exponential moment bound,
\begin{equation}\label{ineq exponential}
\sup_{t\in[0,\infty)}\int_{\R^{3}}\exp\left(\xi_{r_0,\gamma}|v|^\frac{4}{\gamma+2}\right)f_{N,1}(t,v)\d v<C.    
\end{equation}  
\end{theorem}
We refer to Definition \ref{def fN} for the precise notion of weak solutions to the Liouville equation \eqref{eq Liouville}, and to Proposition \ref{prop well-posedness of fN} for the corresponding well-posedness result.

The first rigorous derivation of the Landau equation from many-particles is obtained by \cite{fontbona2009measurability}, where the authors obtained the convergence for \textit{Maxwellian molecules} (corresponding to $\gamma=0$ in \eqref{def AB}) via optimal transport techniques. The result was later generalised in \cite{fournier2009particle} with an improved convergence rate by a coupling method. The propagation of chaos result for the Landau equation with \textit{moderately soft potentials} is shown in \cite{fournier2016propagation}. There, the authors obtained convergence rates for the case $\gamma\in(-1,0)$ by coupling method, and  qualitative convergence for $\gamma\in(-2,-1]$ by the martingale approach. More recently, the relative entropy method has been applied to derive the Landau-like equations (with a uniformly positive-definite matrix $A$) in \cite{carrillo2024mean}. A quantitative relative entropy bound was later obtained in \cite{carrillo2024relative} for the Landau equation for Maxwellian molecules. Note that all of the above results concern \textit{non-conservative} particle models, in which total momentum and energy are conserved only in expectation, but not almost surely.

Unlike non-conservative particles, the so called \textit{Kac particles}, first proposed in \cite{kac1956foundations}, conserve total momentum and energy almost surely. This model was first used for the Landau equation in \cite{Miot2011kac}. In \cite{carrapatoso2016propagation}, the author constructed the corresponding SDE systems for Kac particles \eqref{kacparticles} and established the quantitative and uniform-in-time propagation of
chaos for the Landau equation with Maxwellian molecules. Later,  \cite{fournier2017kac} extended the analysis to hard potentials, where the authors showed the well-posedness of Kac particles \eqref{kacparticles} for $\gamma\in[0,1]$ and proved the quantitative propagation of chaos by coupling method. For soft potentials $\gamma\in[-3,0)$ especially the \textit{Coulomb} case $\gamma=-3$, compactness results for Kac particles  was obtained in \cite{carrillo2025fisher}, where the authors showed the BBGKY hierarchy converges to the infinite Landau hierarchy strongly in $L^1$. Very recently, the remaining difficulty in treating \textit{very soft potentials} has been overcome in \cite{tabary2025propagation} and \cite{feng2025kac}. The former proves propagation of chaos for $\gamma\in(-3,-2]$ via a compactness argument, while the latter resolves the Coulomb case by the duality method. In recent years, there is a noticeable advances in derivation of  the space-homogenous Landau equation under the mean-field scaling, which covers the full range $\gamma \in [-3, 1]$.

In this work, we revisit the mean-field derivation of Landau equation with hard potentials by BBGKY hierarchy method, where we not only derive the infinite hierarchy as \cite{Miot2011kac,carrillo2025fisher}, but also show that only tensorised solutions are allowed. It is the first time that BBGKY hierarchy method has been implemented to prove the propagation of chaos for the Landau equation. 

In order to show propagation of chaos, the first step is to get the uniform estimates of the Liouville equation associated with Kac particle system, which allow us to pass to the limit from the BBGKY hierarchy \eqref{eq BBGKY} to the corresponding infinite Landau hierarchy \eqref{eq Landau hierarchy}.
\begin{theorem}[Existence]
\label{thm existence}
Let $m\in\N$ and $T>0$ be fixed. There exists a subsequence of the weak solutions $\{f_{N',m}\}_{N'\in\N}$ of the BBGKY hierarchy \eqref{eq BBGKY} converging to a weak solution $f_m$ of the Landau hierarchy \eqref{eq Landau hierarchy} defined in Definition \ref{def Landau hierarchy}, such that
$$
f_{N',m}\to f_m\quad\text{in}\quad L^1(0,T;\P(\R^{3m})).
$$
\end{theorem}

The second step is to control the Wasserstein distance of two solutions of the infinite Landau hierarchy.  
\begin{theorem}[Stability estimate]\label{thm uniqueness}
Let $m\in\N$ be fixed. Suppose the initial data are tensorised as $f_m(0)=(f^0)^{\otimes m}$ and $\t f_m(0)=(\t f^0)^{\otimes m}$, where $f^0$ and $\t f^0$ are under Assumption \ref{assumption hard kac}. Then, for any $T>0$, the corresponding weak solutions $f_m(T)$ and $\t f_m(T)$ of the Landau 
hierarchy \eqref{eq Landau hierarchy} satisfy the following stability estimate:
\begin{equation}\label{ineq stability}
\W_2\big(f_m(T),\t f_m(T)\big)\leq C(r_0,\gamma)\sqrt{m(1+T)} \left(\W_2\big(f^0,\t f^0\big)\right)^{1-\eta},
\end{equation}
for any $\eta\in(0,1)$, where the constant $C(r_0,\gamma)$ is independent of $T$ and $m$.
\end{theorem}
This stability estimate relies heavily on the exponential moments bound we obtained and implies the uniqueness of weak solutions with exponential moment bound, which allows us to complete the propagation of chaos result via the BBGKY hierarchy method. To the best of our knowledge, this work provides the first rigorous uniqueness result for the Landau hierarchy. Note that the uniqueness of infinite Landau hierarchy is a much stronger result than the uniqueness of the Landau equation itself. 

We recall the existence and exponential moment estimates for weak solutions of the Landau equation with hard potentials \eqref{eq Landau}:
\begin{proposition}[{\cite[Proposition 5]{fournier2017kac}}]
Under Assumption \ref{assumption hard kac}, there exists a weak solution of \eqref{eq Landau} such that for any $\vphi\in C_b^2(\R^3)$, it holds
\begin{equation}\label{eq weak Landau}
\begin{aligned}
\int_{\R^{3}}\&\vphi(v^1) f(T,v^1)\d v^1-\int_{\R^{3}}\vphi(v^1) f^0(v^1)\d v^1\\=\&\int_0^T\int_{\R^{6}}A(v^1-v^2):\nabla_{v^1v^1}^2\vphi(v^1)f(t,v^1)f(t,v^2)\d v^1\d  v^2\d t\\\&+2\int_0^T\int_{\R^{6}}B(v^1-v^2)\cdot \nabla_{v^1}\vphi(v^1)f(t,v^1)f(t,v^2)\d v^1\d v^2\d t.
\end{aligned}
\end{equation}
Moreover, for any $\beta\in(0,2)$ and constant $\xi_0>0$, weak solutions propagate the exponential moment bound
\begin{equation}\label{ineq exponential Landau}
\sup_{t\in[0,\infty)}\int_{\R^3}\exp(\xi_0|v|^\beta)f(t,v)\d v<\infty.    
\end{equation}
\end{proposition}

We can see that the tensorisation of the Landau equation \eqref{eq Landau} actually gives a weak solution of the Landau hierarchy \eqref{eq Landau hierarchy} defined in Definition \ref{def Landau hierarchy}. That is to say, for any $m\in\N$, $f^{\otimes m}(V_m)=f(v^1)\cdots f(v^m)$ and $f^{\otimes (m+1)} (V_{m+1})=f(v^1)\cdots f(v^{m+1})$ with tensorised initial data solve the $m$-th equations in the hierarchy. We deduce from Theorem \ref{thm existence} that if the weak solution of the Landau hierarchy \eqref{eq Landau hierarchy} obtained from the BBGKY hierarchy and the tensorised weak solution given by the Landau equation \eqref{eq Landau} have the same tensorised initial data, then these two weak solutions coincide at any time $t>0$. It yields the following propagation of chaos from Theorem \ref{thm exponential} and Theorem \ref{thm existence}. 
\begin{theorem}[Propagation of Chaos]\label{thm poc}
 For any fixed $m\in\N$, the sequence $\{f_{N,m}\}_{N\in\N}$, which are $m$-th marginals of weak solutions of the Liouville equation 
\eqref{eq Liouville}, converges to the unique weak solution $f_m$ of the Landau hierarchy \eqref{eq Landau hierarchy}, which has the tensorised form such as $f_m=f^{\otimes m}$. In particular, the propagation of chaos holds, 
$$
f_{N,m}\to f^{\otimes m}\quad\text{in}\quad L^1(0,T;\P(\R^{3m}))\quad \text{as}\quad N\to\infty.
$$
\end{theorem}
We conclude with some remarks on the assumptions imposed on the initial data.
\begin{remark}\label{rmk initial}
Under Assumption \ref{assumption hard kac}, the conservative particle system \eqref{kacparticles} is wellposed; see \cite{fournier2017kac} for the proof for both Maxwellian molecules and hard potentials, even under weaker assumptions.

The first condition in Assumption \ref{assumption hard kac}, namely, the compact support of the initial data, is stronger than those typically found in the literature. This condition is used to constrain the initial configuration of the particles by ensuring that, for any arbitrarily fixed $r_0>0$
$$
\frac{1}{N}\sum_{i=1}^N|V_0^i|^2\leq r_0^2\quad\text{a.s.}
$$
This bound is important for getting the moment estimate at the particle level. Note that it also implies the standard physical properties required of the initial data, such as finite mass, momentum and energy.

Unlike in \cite{carrillo2025fisher}, we do not require the initial data to have finite Fisher information. The finite entropy condition suffices to construct a weak solution of the Landau hierarchy \eqref{eq Landau hierarchy} with hard potentials.
\end{remark}

This work is organised as follows: In Section \ref{sec moment}, we show the crucial estimate, namely, the uniform-in-time and uniform-in-$N$ propagation of exponential moment for the Kac particles. Section \ref{sec existence} is devoted to proving the well-posedness of the BBGKY hierarchy \eqref{eq BBGKY} and the existence of weak solutions of the infinite Landau hierarchy \eqref{eq Landau hierarchy}. In Section \ref{sec uniqueness}, we obtain the stability estimate as well as the uniqueness of weak solutions of the infinite Landau hierarchy with tensorised initial data, which completes the BBGKY hierarchy method for showing propagation of chaos.

\section{Uniform-in-time exponential moment estimate}\label{sec moment}
In this section, we prove Theorem \ref{thm exponential}, which gives a bound on the exponential moment of the first marginal of the solution of the Liouville equation. The proof is based on a summation over polynomial moments, using a Taylor expansion argument to control the exponential moment. The key technical step involves proving uniform polynomial moment bounds, as stated Proposition \ref{prop polynomial}. The proof relies on a sharpened Povzner-type inequality, along with a decomposition of the time interval into short- and long-time regimes.  
\begin{proposition}\label{prop polynomial} 
Let $f_N$ be the weak solution of the Liouville equation \eqref{eq Liouville}, where $ f_N = \law(V^1,\ldots,V^N) $ corresponds to the particle system \eqref{kacparticles}, and assume the initial data satisfies Assumption \ref{assumption hard kac}. Then, for any $p\geq 4$, the $p$-moment estimate of the first marginal $f_{N,1}$ satisfies the following uniform-in-time and uniform-in-$N$ bound:
\begin{equation}\label{ineq polynomial}
\sup_{t\in[0,\infty)}\int_{\R^{3}}|v|^pf_{N,1}(t,v)\d v\leq \left(C_{r_0,\gamma}\right)^pp^{\frac{(2+\gamma)(p-2)}{4}},
\end{equation}
where the constant $C_{r_0,\gamma}$ is independent of $p$, $t$ and $N$.
\end{proposition}

\begin{remark}[Sharpness of the exponent]\label{rmk sharp}
In the Maxwellian case, i.e., when $\gamma=0$, the propagation of polynomial moment was obtained in \cite{fournier2017kac} as 
\begin{equation}\label{ineq maxwellian}
\sup_{t\in[0,\infty)}\int_{\R^3}|v|^pf_{N,1}(t,v)\d v\leq\int_{\R^3}|v|^pf^0(t,v)\d v\leq  (2r_0^2p)^\frac{p}{2},    
\end{equation}
where the second inequality follows Assumption \ref{assumption hard kac}. In the case $\gamma\in(0,1]$, it is  showed in \cite{fournier2017kac} that 
$$
\sup_{t\in[0,\infty)}\int_{\R^3}|v|^pf_{N,1}(t,v)\d v\leq (C_pr_0^{p+\gamma})^\frac{p}{p+\gamma}.
$$
However, in order to obtain exponential moment estimates rather than polynomial  ones, the dependence of the constant on $p$ becomes crucial. In our proof, we make this dependence explicit by showing a sharper version of Povzner-type of inequality and carefully tracking the exponents. The exponent in the estimate \eqref{ineq polynomial} is believed to be almost sharp. This can be seen by taking $\gamma=0$ in \eqref{ineq polynomial}, which then recovers the estimate in the Maxwellian case \eqref{ineq maxwellian}, up to a lower order multiplicative factor.
\end{remark}

We postpone the proof of Proposition \ref{prop polynomial} until after showing the exponential estimate \eqref{ineq exponential}. The proof is as follows.

\begin{proof}[Proof of Theorem \ref{thm exponential}]
For some $\xi>0$ which will be determined later, we consider the exponential moment $\exp(\xi|v|^\frac{4}{2+\gamma})$ with $\gamma\in(0,1]$. By using the Taylor's expansion, it holds for any $t\geq 0$
$$
\begin{aligned}
\int_{\R^{3}}\exp\&(\xi|v^1|^\frac{4}{2+\gamma})f_{N,1}(t,v^1)\d v^1\\=\&\int_{\R^{3}}f_{N,1}(t,v^1)\d v^1+\xi\int_{\R^{3}}|v^1|^\frac{4}{2+\gamma}f_{N,1}(t,v^1)\d v^1+\frac{\xi^2}{2}\int_{\R^{3}}|v^1|^\frac{8}{2+\gamma}f_{N,1}(t,v^1)\d v^1\\\&+\sum_{n=3}^\infty\frac{\xi^n}{n!}\int_{\R^{3}}|v^1|^\frac{4n}{2+\gamma}f_{N,1}(t,v^1)\d v^1\\\leq\& C(\xi,\gamma)+\sum_{n=3}^\infty\frac{\xi^n}{n!}\left(C_{r_0,\gamma}\right)^\frac{4n}{2+\gamma}(\frac{4n}{2+\gamma})^{n-\frac{2+\gamma}{2}}\\\leq\&C(\xi,\gamma)+\sum_{n=3}^\infty\frac{1}{n!}\left[\xi\left(C_{r_0,\gamma}\right)^\frac{4}{2+\gamma}\frac{4n}{2+\gamma}\right]^{n},
\end{aligned}
$$
where we used the estimate \eqref{ineq polynomial} at the second last inequality with $p=\frac{4n}{2+\gamma}$, and the constant $C(\xi,\gamma)$ is independent of $t$. Applying the Stirling's formula, the last summation above is bounded as long as  
$$
0<\xi\left(C_{r_0,\gamma}\right)^\frac{4}{2+\gamma}\frac{4}{2+\gamma}<\frac{1}{e},\quad\text{i.e.,}\quad 0<\xi<\frac{2+\gamma}{4e\left(C_{r_0,\gamma}\right)^\frac{4}{2+\gamma}}.
$$
Therefore, there exits some $\xi=\xi_{r_0,\gamma}$ such that 
$$
\sup_{t\in[0,\infty)}\int_{\R^{3}}\exp(\xi_{r_0,\gamma}|v^1|^\frac{4}{2+\gamma})f_{N,1}(t,v^1)\d v^1<\infty.
$$
\end{proof}
The remaining of this section devotes to proof the polynomial moment estimate (Proposition \ref{prop polynomial}). We denote the Lebesgue measure on $\R^{3N}$ by $\d V_N := \d v^1 \d v^2 \cdots \d v^N$.

\begin{lemma}\label{lemma weak form}
Assume $f_N$ is the weak solution of the Liouville equation \eqref{eq Liouville} defined in Definition \ref{def fN}. Then, for any $p\geq 2$, the estimate holds
\begin{equation}\label{ineq integral p}
\begin{aligned}
\int_{\R^{3N}}\& f_N(t)\frac{1}{N}\sum_{i=1}^N|v^i|^p\d V_N-\int_{\R^{3N}} (f^0)^{\otimes N}\frac{1}{N}\sum_{i=1}^N|v^i|^p\d V_N\\
\leq\&\frac{p^2}{2N^2}\sum_{i\neq j}\int_0^t\int_{\R^{3N}}\Big[|v^i|^{p-2}|v^j|^{2}+|v^j|^{p-2}|v^i|^{2}\Big]|v^i-v^j|^{\gamma}f_N(s)\d V_N\d s\\
\&-\frac{p}{N^2}\sum_{i\neq j}\int_0^t\int_{\R^{3N}}\Big[|v^i|^p+|v^j|^p\Big]|v^i-v^j|^{\gamma}f_N(s)\d V_N\d s.
\end{aligned}
\end{equation}
\begin{proof}[Proof of Lemma \ref{lemma weak form}]
Every term above is well-defined since, for any $p\geq 2$ and $t\geq 0$, recalling the corresponding particle system \eqref{kacparticles}, we obtain
$$
\begin{aligned}
\int_{\R^{3N}}f_N(t)\frac{1}{N}\sum_{i=1}^N|v^i|^p\d V_N\leq\&  N^{\frac{p}{2}-1}\int_{\R^{3N}} f_N(t)\left(\sum_{i=1}^N|v^i|^2\right)^\frac{p}{2}\d V_N=N^{\frac{p}{2}-1}\E\left[\left(\sum_{i=1}^N|V_t^i|^2\right)^\frac{p}{2}\right] \\  =\& N^{\frac{p}{2}-1}\E\left[\left(\sum_{i=1}^N|V_0^i|^2\right)^\frac{p}{2}\right]\leq N^{p-1}r_0^p,
\end{aligned}
$$
where the second inequality follows from the conservations \eqref{eq conservation} and the last inequality uses Assumption \ref{assumption hard kac}, which ensures
$$
\frac{1}{N}\sum_{i=1}^N|V_0^i|^2\leq r_0^2,\quad\text{a.s.}
$$
For any fixed $N\in\N$, let $\vphi_N:=\frac{1}{N}\sum_{i=1}^N|v^i|^p$. Plugging this test function into the weak form \eqref{eq weak Liouville}, we obtain
\begin{equation}\label{eq p}
\begin{aligned}
\int_{\R^{3N}}\&f_N(t)\frac{1}{N}\sum_{i=1}^N|v^i|^p\d V_N-\int_{\R^{3N}}(f^0)^{\otimes N}\frac{1}{N}\sum_{i=1}^N|v^i|^p\d V_N\\
=\&\frac{p^2}{2N^2}\sum_{i\neq j}\int_0^t\int_{\R^{3N}}\Big[|v^i|^{p-2}|v^j|^{2}+|v^j|^{p-2}|v^i|^{2}\Big]|v^i-v^j|^{\gamma}f_N(s)\d V_N\d s\\&-\frac{p(p-2)}{2N^2}\sum_{i\neq  j}\int_0^t\int_{\R^{3N}}(v^i\cdot v^j)^2(|v^{j}|^{p-4}+|v^{i}|^{p-4})|v^i-v^j|^{\gamma}f_N(s)\d V_N\d s\\
\&-\frac{p}{N^2}\sum_{i\neq j}\int_0^t\int_{\R^{3N}}\Big[|v^i|^p+|v^j|^p\Big]|v^i-v^j|^{\gamma}f_N(s)\d V_N\d s,
\end{aligned}
\end{equation}  
where the desired inequality \eqref{ineq integral p} follows from the non-positivity of the second term on the right-hand side.
\end{proof}
\end{lemma}
\begin{remark}
Under our assumptions, we  immediately obtain that for any convex and increasing function $\psi:\R^3\to\R$ and for all $t\geq 0$, it holds almost surely that
$$
\psi\left(\frac{1}{N}\sum_{i=1}^N|V_t^i|^2\right)=\psi\left(\frac{1}{N}\sum_{i=1}^N|V_0^i|^2\right)\leq \psi(r_0^2),\quad \text{a.s.}
$$
However, this does not directly yield a uniform-in-$N$ bound estimate of 
$\frac{1}{N}\sum_{i=1}^N\psi\left(|V_t^i|^2\right)$.
\end{remark}
To simplify the presentation below, we will continue to use the time-derivative form on the left-hand side of \eqref{ineq integral p}, even though it is more rigorous to interpret it in its time-integrated form.

In order to get further estimate on \eqref{ineq integral p}, we require a sharpened version of  Povzner-type inequality. The key idea is to extract an upper bound involving quadratic terms, which can be eventually bounded by replacing them with their initial values via the conservation \eqref{eq conservation}.
\begin{lemma}[Sharpened Povzner-type inequality]\label{lemma Povzner}
For $x,y\geq 0$ and $\gamma\in(0,1]$, it holds 
\begin{equation}\label{ineq povzner new}
\begin{aligned}
\&\left(-x^p-y^p+\frac{p}{2}x^{p-2}y^2+\frac{p}{2}y^{p-2}x^2\right)|x-y|^\gamma
\\\leq\& -\frac{1}{2}x^{p+\gamma}-\frac{1}{2}y^{p+\gamma}+x^{p}y^\gamma+y^px^\gamma+p^{1+\frac{\gamma}{2}}(x^{p-2+\gamma}y^{2}+y^{p-2+\gamma}x^2).
\end{aligned}
\end{equation}   
\end{lemma}
\begin{remark}
Comparing to the Povzner-type of estimate obtained in \cite{desvillettes2000spatially1}, namely, 
$$
-x^p-y^p+\frac{p}{2}x^{p-2}y^2+\frac{p}{2}y^{p-2}x^2\leq -\frac{1}{2}x^p-\frac{1}{2}y^p+p^\frac{3}{2}(x^{p-1}y+y^{p-1}x),   
$$
we take account the multiplier $|x-y|^\gamma$ and get a more refined coefficient in terms of $p$.
\end{remark}
\begin{proof}[Proof of Lemma \ref{lemma Povzner}]
The following elementary bounds hold for $x,y\geq 0$ and $\gamma\in(0,1]$:
$$|x-y|^\gamma\geq x^\gamma-y^\gamma,\quad |x-y|^\gamma\geq y^\gamma-x^\gamma \quad\text{and}\quad|x-y|^\gamma\leq x^\gamma+y^\gamma.$$ It implies that
$$
\begin{aligned}
\&\left(-x^p-y^p+\frac{p}{2}x^{p-2}y^2+\frac{p}{2}y^{p-2}x^2\right)|x-y|^\gamma
\\\leq\& -x^{p+\gamma}+x^{p}y^\gamma-y^{p+\gamma}+y^{p}x^\gamma+\frac{p}{2}(x^{p-2+\gamma}y^{2}+x^{p-2}y^{2+\gamma}+y^{p-2+\gamma}x^2+y^{p-2}x^{2+\gamma}).
\end{aligned}
$$    
We separate the following term,
$$
x^{p}y^\gamma+y^{p}x^\gamma+\frac{p}{2}(x^{p-2+\gamma}y^{2}+y^{p-2+\gamma}x^2),
$$
as it has the desired structure for our purposes. We now focus on estimating the remaining terms, namely,
$$
 -x^{p+\gamma}-y^{p+\gamma}+\frac{p}{2}(x^{p-2}y^{2+\gamma}+y^{p-2}x^{2+\gamma}).
$$
For some small $\eps>0$, if $x\leq \eps y$, then it follows that 
$$
\begin{aligned}
-x^{p+\gamma}-y^{p+\gamma}\&+\frac{p}{2}x^{p-2}y^{2+\gamma}+\frac{p}{2}y^{p-2}x^{2+\gamma}\\
\leq \&-x^{p+\gamma}+ \Big[-1+\frac{p}{2}(\eps^{p-2}+\eps^{2+\gamma})\Big]y^{p+\gamma}\\
\leq\&-x^{p+\gamma} -\frac{1}{2}y^{p+\gamma},
\end{aligned}
$$
where the last inequality can be achieved by assuming $\eps^{p-2}+\eps^{2+\gamma}\leq 2\eps ^2\leq \frac{1}{p}$, which is sufficient to let $\eps=\frac{1}{\sqrt{2p}}$. Similarly, if $y\leq \eps x$ with $\eps=\frac{1}{\sqrt{2p}}$, then it provides us that 
$$
-x^{p+\gamma}-y^{p+\gamma}+\frac{p}{2}(x^{p-2}y^{2+\gamma}+y^{p-2}x^{2+\gamma})\leq-y^{p+\gamma} -\frac{1}{2}x^{p+\gamma}.
$$
The remaining case is that: if $\frac{1}{\sqrt{2p}} y<x<\sqrt{2p}y$, we have
$$
-x^{p+\gamma}-y^{p+\gamma}+\frac{p}{2}x^{p-2}y^{2+\gamma}+\frac{p}{2}y^{p-2}x^{2+\gamma}\leq-x^{p+\gamma}-y^{p+\gamma}+\frac{p}{2}(\sqrt{2p})^\gamma\left(x^{p-2+\gamma}y^{2}+y^{p-2+\gamma}x^{2}\right).
$$
In conclusion, it holds
$$
\begin{aligned}
\Big(-x^p-y^p+\&\frac{p}{2}x^{p-2}y^2+\frac{p}{2}y^{p-2}x^2\Big)|x-y|^\gamma
\\\leq\&x^{p}y^\gamma+y^{p}x^\gamma+\frac{p}{2}(x^{p-2+\gamma}y^{2}+y^{p-2+\gamma}x^2)
\\\&+\max\Big\{-x^{p+\gamma} -\frac{1}{2}y^{p+\gamma},\,-y^{p+\gamma} -\frac{1}{2}x^{p+\gamma},\\\&\qquad\qquad-x^{p+\gamma} -y^{p+\gamma}+\frac{p}{2}(\sqrt{2p})^\gamma\left(x^{p-2+\gamma}y^{2}+y^{p-2+\gamma}x^{2}\right)\Big\}
\\\leq\& -\frac{1}{2}x^{p+\gamma}+x^{p}y^\gamma-\frac{1}{2}y^{p+\gamma}+y^px^\gamma+p^{1+\frac{\gamma}{2}}(x^{p-2+\gamma}y^{2}+y^{p-2+\gamma}x^2).
\end{aligned}
$$

\end{proof}
Applying Lemma \ref{lemma Povzner}, we resume our estimate \eqref{ineq integral p} as follows:
\begin{equation}\label{ineq afterPovnzer}
\begin{aligned}
\&\frac{\d}{\d t} \int_{\R^{3N}}f_N(t,V_N)\frac{1}{N}\sum_{i=1}^N|v^i|^p\d V_N\\
\leq\&\frac{p}{N^2}\sum_{i\neq j}\int_{\R^{3N}}\Big[-|v^i|^p-|v^j|^p+\frac{p}{2}|v^i|^{p-2}|v^j|^{2}+\frac{p}{2}|v^j|^{p-2}|v^i|^{2}\Big]|v^i-v^j|^{\gamma}f_N\d V_N\\
\leq\&\frac{p}{N^2}\sum_{i\neq j}\int_{\R^{3N}}\Big[-\frac{1}{2}|v^i|^{p+\gamma}-\frac{1}{2}|v^j|^{p+\gamma}+|v^i|^{p}|v^j|^\gamma+|v^j|^p|v^i|^\gamma\\&\qquad\qquad\qquad\qquad+p^{1+\frac{\gamma}{2}}(|v^i|^{p-2+\gamma}|v^j|^{2}+|v^j|^{p-2+\gamma}|v^i|^2)\Big]f_N\d V_N.
\end{aligned}
\end{equation}
We rewrite the moments into the expectation form by exchangeability such as
\begin{equation}\label{eq pde-sde form}
\int_{\R^{3N}}f_N(t,V_N)\frac{1}{N}\sum_{i=1}^N|v^i|^p\d V_N=\int_{\R^{3N}}f_{N,1}(t,v^1)|v^1|^p\d v^1=\E[|V_t^1|^p]
\end{equation}
and
$$
\begin{aligned}
\int_{\R^{3N}}f_N(t,V_N)\frac{1}{N^2}\sum_{i\neq j}|v^i|^p|v^i|^\gamma\d V_N=\&\frac{N(N-1)}{N^2}\int_{\R^{3N}}f_{N,2}(t,v^1,v^2)|v^1|^p|v^2|^\gamma\d v^1\d v^2\\\leq\&\E[|V_t^1|^p|V_t^2|^\gamma],
\end{aligned}
$$
where it also holds $\E[|V_t^1|^p|V_t^2|^{\gamma}]\leq \E[|V_t^1|^{p+\gamma}]$. Then, \eqref{ineq afterPovnzer} yields
$$
\begin{aligned}
\frac{\d}{\d t}\E\&[|V_t^1|^p]\leq -p\E[|V_t^1|^{p+\gamma}]+2p\E[|V_t^1|^p|V_t^2|^{\gamma}]+2p^{2+\frac{\gamma}{2}}\E[|V_t^1|^{p-2+\gamma}|V_t^2|^2]\\\leq\& -p\E[|V_t^1|^{p+\gamma}]+2p\E[|V_t^1|^p\frac{1}{N}\sum_{i=1}^N|V_t^i|^{\gamma}]+2p^{2+\frac{\gamma}{2}}\E[|V_t^1|^{p-2+\gamma}\frac{1}{N}\sum_{i=1}^N|V_t^i|^2]\\
\leq\& -p\E\left[|V_t^1|^{p+\gamma}\right]+2p\E\left[|V_t^1|^p\left(\frac{1}{N}\sum_{i=1}^N|V_t^i|^2\right)^{\gamma/2}\right]\\\&\qquad\qquad\qquad\qquad\qquad\qquad+2p^{2+\frac{\gamma}{2}}\E\left[|V_t^1|^{p-2+\gamma}\left(\frac{1}{N}\sum_{i=1}^N|V_t^i|^2\right)\right].  
\end{aligned}
$$
The almost sure conservation \eqref{eq conservation} avoids using the H\"{o}lder's inequality and yields 
\begin{equation}\label{ineq original power}
\begin{aligned}
\frac{\d}{\d t}\E[|V_t^1|^p]
\leq\& -p\E\left[|V_t^1|^{p+\gamma}\right]+2p\E\left[|V_t^1|^p\left(\frac{1}{N}\sum_{i=1}^N|V_0^i|^2\right)^{\gamma/2}\right]\\\&\qquad\qquad\qquad\qquad\qquad\qquad+2p^{2+\frac{\gamma}{2}}\E\left[|V_t^1|^{p-2+\gamma}\left(\frac{1}{N}\sum_{i=1}^N|V_0^i|^2\right)\right]\\\leq\& -p\E\left[|V_t^1|^{p+\gamma}\right]+2pr_0^\gamma\E\left[|V_t^1|^p\right]+2p^{2+\frac{\gamma}{2}}r_0^2\E\left[|V_t^1|^{p-2+\gamma}\right],
\end{aligned}
\end{equation}
where the initial distribution $f^0$ of each $V_0^i$ is compactly supported in the ball of radius $r_0$ under Assumption \ref{assumption hard kac}.

If we let $p=2$ in \eqref{eq p}, then it has
$$
\frac{\d}{\d t}\int_{\R^{3N}}f_N(t,V_N)\frac{1}{N}\sum_{i=1}^N|v^i|^2\d V_N=\frac{\d}{\d t}\int_{\R^3}|v^1|^2f_{N,1}(t, v^1)\d v^1=0.
$$
Thus, the normalised measure $$\frac{|v^1|^2f_{N,1}(t,v^1)}{\int_{\R^3}|v|^2f^0\d v}\d v^1$$ is a probability measure for any $t\geq0$. This implies that, for $p>2$,
$$
\begin{aligned}
\E\left[|V_t^1|^{p}\right]\leq \left(\E\left[\left((\int_{\R^3}|v|^2f^0\d v)|V_t^1|^{p-2}\right)^\frac{p-2+\gamma}{p-2}\frac{|V_t^1|^2}{\int_{\R^3}|v|^2f^0\d v}\right]\right)^\frac{p-2}{p-2+\gamma}\leq \left(\E\left[|V_t^1|^{p+\gamma}\right]\right)^\frac{p-2}{p-2+\gamma};
\end{aligned}
 $$
similarly, for $p>4-\gamma$
$$
\begin{aligned}
\&\E\left[|V_t^1|^{p-2+\gamma}\right]\leq\left(\E\left[\left(|V_t^1|^{p-4+\gamma}\right)^\frac{p-2}{p-4+\gamma}|V_t^1|^2\right]\right)^\frac{p-4+\gamma}{p-2}\leq \left(\E\left[|V_t^1|^{p}\right]\right)^\frac{p-4+\gamma}{p-2}.
\end{aligned}
 $$
Plugging these estimates into \eqref{ineq original power} yields the following lemma.
\begin{lemma}
Under the same assumptions as in Proposition \ref{prop polynomial}, then the following moment estimate holds for all $p>4-\gamma$,
$$
\begin{aligned}
\frac{\d}{\d t}\E[|V_t^1|^p]
\leq\& -p\left(\E[|V_t^1|^{p}]\right)^{1+\frac{\gamma}{p-2}}+2pr_0^\gamma\E[|V_t^1|^p]+2r_0^2p^{2+\frac{\gamma}{2}}\left(\E[|V_t^1|^p]\right)^{1-\frac{2-\gamma}{p-2}}. 
\end{aligned}
$$     
\end{lemma}

\medskip

We remain to analyse the ordinary differential inequality
\begin{equation}\label{ineq ht}
 \frac{\d}{\d t}h_t\leq -a(h_t)^{1+\alpha}+bh_t+c(h_t)^{1-\beta},   
\end{equation}
with $h_t=\E[|V_t^1|^p]$ and
$$a=p,\quad b=2pr_0^\gamma,\quad c=2r_0^2p^{2+\frac{\gamma}{2}},\quad \alpha=\frac{\gamma}{p-2},\quad \beta=\frac{2-\gamma}{p-2}.$$
The key idea is to decompose $t\in[0,\infty)$ into short- and long-time regimes.  
Firstly, we can control $h_t$ on the right-hand side by $(h_t)^{1+\alpha}$ as follows
$$
\begin{aligned}
 \frac{\d}{\d t}h_t\leq\& -a(h_t)^{1+\alpha}+bh_t+c(h_t)^{1-\beta}\\
 \leq\& -a(h_t)^{1+\alpha}+\frac{(\delta h_t)^{1+\alpha}}{1+\alpha}+\frac{\alpha(\frac{b}{\delta})^\frac{1+\alpha}{\alpha}}{1+\alpha}+c(h_t)^{1-\beta}
 \\
 \leq\&\frac{\alpha(\frac{b}{a^{1/(1+\alpha)}})^\frac{1+\alpha}{\alpha}}{1+\alpha}+c(h_t)^{1-\beta}\\
 \leq\&\frac{2}{p}\left(\frac{2pr_0^\gamma}{p^\frac{p-2}{p-2+\gamma}}\right)^\frac{p-2+\gamma}{\gamma}+c(h_t)^{1-\beta}=2^\frac{p-2+2\gamma}{\gamma}r_0^{p-2+\gamma}+2r_0^2p^{2+\frac{\gamma}{2}}(h_t)^{1-\beta},
\end{aligned}
 $$
where we take $\delta=a^\frac{1}{1+\alpha}$
and use the elementary inequality 
 $$
 \frac{\alpha}{1+\alpha}=\frac{\gamma/(p-2)}{1+\gamma/(p-2)}\leq\frac{2}{p}.
 $$
The inequality implies that
$$
\frac{\d}{\d t}h_t^\beta=\beta\frac{\frac{\d}{\d t}h_t}{(h_t)^{1-\beta}}\leq \beta\left(\frac{2^\frac{p-2+2\gamma}{\gamma}r_0^{p-2+\gamma}}{(h_t)^{1-\beta}}+2r_0^2p^{2+\frac{\gamma}{2}}\right).
$$
Then it holds either 
$$
\frac{2^\frac{p-2+2\gamma}{\gamma}r_0^{p-2+\gamma}}{(h_t)^{1-\beta}}\geq 2r_0^2p^{2+\frac{\gamma}{2}}\quad\text{or}\quad \frac{2^\frac{p-2+2\gamma}{\gamma}r_0^{p-2+\gamma}}{(h_t)^{1-\beta}}\leq 2r_0^2p^{2+\frac{\gamma}{2}}.
$$
That is to say, we have either
$$
h_t\leq \left(\frac{2^\frac{p-2+2\gamma}{\gamma}r_0^{p-2+\gamma}}{2r_0^2p^{2+\frac{\gamma}{2}}}\right)^\frac{1}{1-\beta}=r_0^{p-2}\left(\frac{2^\frac{p-2+\gamma}{\gamma}}{p^{2+\frac{\gamma}{2}}}\right)^\frac{p-2}{p-4+\gamma}\leq r_0^{p-2}\left(2^{\frac{p-2}{\gamma}-3-\gamma} \right)^\frac{p-2}{p-4+\gamma}\leq (r_04^\frac{1}{\gamma})^{p-2},
$$
with $p\geq4$; or
$$
\begin{aligned}
\frac{\d}{\d t}h_t^\beta\leq   4r_0^2p^{2+\frac{\gamma}{2}}\beta\leq\frac{4-2\gamma}{p}\left(4r_0^2p^{2+\frac{\gamma}{2}}\right)\leq 16r_0^2p^{1+\frac{\gamma}{2}}.    
\end{aligned}
$$
Integrating over time and using the initial bound $h_0=\E[V_0^1|^p]\leq r_0^p$, we obtain
$$
\begin{aligned}
h_t^\beta\leq\& \max\Big\{r_0^{p\beta}+16tr_0^2p^{1+\frac{\gamma}{2}},(r_04^\frac{1}{\gamma})^{\beta(p-2)}\Big\}.
\end{aligned}
$$
Now consider the small-time regime defined by
\begin{equation}\label{tleq}
  0\leq t< p^{-\frac{\gamma(2+\gamma)}{4}}.
\end{equation}
In this case, we obtain, for some constant $C_{r_0,\gamma}$,
$$
\begin{aligned}
h_t^\beta\leq \max\Big\{r_0^{4-2\gamma}+16r_0^2p^{1+\frac{\gamma}{2}-\frac{\gamma(2+\gamma)}{4}},(r_04^\frac{1}{\gamma})^{2-\gamma}\Big\}\leq C_{r_0,\gamma}p^{1-\frac{\gamma^2}{4}}.
\end{aligned}
$$
Then,  we have the estimate for small-time regime
\begin{equation}\label{ineq small t}
\sup_{t\in[0, p^{-\frac{\gamma(2+\gamma)}{4}})}h_t\leq \left(C_{r_0,\gamma}p^{1-\frac{\gamma^2}{4}}\right)^\frac{p-2}{2-\gamma}=\left(C_{r_0,\gamma}\right)^\frac{p-2}{2-\gamma}p^{\frac{(2+\gamma)(p-2)}{4}}.
\end{equation}

To proceed with the estimate in the large-time, we apply the lemma below; see, for instance \cite[Proposition 6]{fournier2021stability}.
\begin{lemma}
For $t>0$, let $h_t\geq 0$ satisfying \begin{equation}\label{ineq ut}
 \frac{\d}{\d t}h_t\leq -a(h_t)^{1+\alpha}+bh_t+c(h_t)^{1-\beta},   
\end{equation}
then it holds
 $$
h_t\leq \left(\frac{2}{a\alpha t}\right)^\frac{1}{\alpha} +\left(\frac{4b}{a}\right)^\frac{1}{\alpha} +\left(\frac{4c}{a}\right)^\frac{1}{\alpha+\beta}. 
 $$
\end{lemma} 
When $t$ belongs to the complementary set of \eqref{tleq}, namely,
\begin{equation}\label{tgeq}
  \frac{1}{t}\leq p^\frac{\gamma(2+\gamma)}{4},
\end{equation}
it has
$$
\begin{aligned}
h_t\leq\& \left(\frac{2(p-2)}{p\gamma t}\right)^\frac{p-2}{\gamma} +(8r_0^\gamma)^\frac{p-2}{\gamma} +\left(8r_0^2p^{1+\frac{\gamma}{2}}\right)^\frac{p-2}{2}\\
\leq\& \left(\frac{4}{\gamma t}\right)^\frac{p-2}{\gamma} +(8r_0^\gamma)^\frac{p-2}{\gamma} +\left(8r_0^2p^{1+\frac{\gamma}{2}}\right)^\frac{p-2}{2}\\
\leq\& \left(C_{r_0,\gamma}\,p^{\frac{2+\gamma}{2}}\right)^\frac{p-2}{2}, 
\end{aligned}
$$
where the constant only depends on $r_0$ and $\gamma$.

\medskip

Together with the estimate for small time \eqref{ineq small t},
we thus obtain the uniform-in-time and uniform-in-$N$ estimate, by using the equivalence \eqref{eq pde-sde form},
$$
\sup_{t\in[0,\infty)}h_t=\sup_{t\in[0,\infty)}\E[|V_t^1|^p]=\sup_{t\in[0,\infty)}\int_{\R^{3}}|v^1|^pf_{N,1}(t,v^1)\d v^1\leq \left(C_{r_0,\gamma}\right)^pp^{\frac{(2+\gamma)(p-2)}{4}},
$$
where $C_{r_0,\gamma}$ is some constant independent of $p$, $t$ and $N$. This concludes the proof of Proposition \ref{prop polynomial}.

\section{Existence of solutions of infinite Landau hierarchy}\label{sec existence}

In this section, we first establish the well-posedness result of the Liouville equation \eqref{eq Liouville}. Furthermore, we prove the marginal of the Liouville equation \eqref{eq Liouville} converges to a weak solution of the Landau hierarchy \eqref{eq Landau hierarchy}, as stated in Theorem \ref{thm existence}.

We define weak solutions of the Liouville equation \eqref{eq Liouville} as follows.
\begin{definition}[Weak solutions of the Liouville equation]\label{def fN} 
For any $T>0$, a weak solution of the Liouville equation \eqref{eq Liouville} with $f_{N}(0)=(f^0)^{\otimes N}$ satisfies the following
weak form. For any $\vphi_{N}\in C^2_b(\R^{3N})$, it holds
\begin{equation}\label{eq weak Liouville}
\begin{aligned}
\int_{\R^{3N}}f_{N}\& (T)\vphi_{N}\d V_{N}-\int_{\R^{3N}} (f^0)^{\otimes N}\vphi_{N}\d V_{N}\\=\&\frac{1}{N}\sum_{i\neq j}^N\int_0^T\int_{\R^{3N}}A(v^i-v^j):(\nabla_{v^iv^i}^2\vphi_{N}-\nabla_{v^iv^j}^2\vphi_{N})f_{N}\d V_{N}\d t\\\&+\frac{1}{N}\sum_{i\neq j}^N\int_0^T\int_{\R^{3N}}B(v^i-v^j)\cdot (\nabla_{v^i}\vphi_{N}-\nabla_{v^j}\vphi_{N}) f_{N}\d V_{N}\d t.
\end{aligned}
\end{equation}
Moreover, $f_N\in L^\infty([0,T],L^1(\R^{3N}))$ represents the density of the joint distribution of the particles defined in \eqref{kacparticles} such that $\law(V^1_t,\ldots,V^N_t)=f_N(t)\d V_N$ and it has the finite second moment and finite entropy. 
\end{definition}
The well-posedness of Liouville equation \eqref{eq Liouville} is  presented as below.
\begin{proposition}[Well-posedness of the Liouville equation]\label{prop well-posedness of fN}
Under Assumption \ref{assumption hard kac}, there exists a unique weak solution of the Liouville equation \eqref{eq Liouville}. 
\end{proposition}
\begin{proof}[Proof of Proposition \ref{prop well-posedness of fN}]
We will show that there exists a weak solution of the Liouville equation \eqref{eq Liouville}, which can be approximated by solutions of the equations with the truncation of relative velocity and the vanishing viscosity. We introduce $f_N^\eps$ which satisfies
\begin{equation}\label{eq vanishing}
\begin{aligned}
\p_tf_N^\eps=\frac{1}{2N}\sum_{i\neq j}\div_{v^i-v^j}\&\left[|v^i-v^j|^{\gamma+2}\Pi(v^i-v^j)I_{|v^i-v^j|\leq\frac{1}{\eps}}(\nabla_{v^i}-\nabla_{v^j})f_N^\eps\right]\\\&+\eps\sum_{i=1}^N\Delta_{v^i}f_N^\eps,      
\end{aligned}
\end{equation}
with the continuous bounded initial data $f_{N}^\eps=(f^{\eps,0})^{\otimes N}$.
For each $\eps>0$, this is a linear non-degenerate parabolic equation and the well-posedness of \eqref{eq vanishing} in $C^1\left([0,\infty),C_b^2(\R^{3N})\right)$ can be obtained by classical theory. We assume that the sequence $f^{\eps,0}\to f^0$ strongly in  $L^1$ with and its entropy and second moment is uniformly-in-$\eps$ bounded
$$
\int_{\R^{3}}f^{\eps,0}(v^1)\log f^{\eps,0}(v^1)\d v^1<\infty, \quad
\int_{\R^{3}}f^{\eps,0}(v^1)|v^1|^2\d v^1<\infty.
$$
For any $T>0$, we test \eqref{eq vanishing} against $\log f_N^\eps$ and get
$$
\begin{aligned}
\frac{1}{N}\int_{\R^{3N}}\&f_N^\eps(T)\log f_N^\eps(T)\d V_N=\int_{\R^{3}}f^{\eps,0}(v^1)\log f^{\eps,0}(v^1)\d v^1\\\&-\frac{1}{2N^2}\sum_{i\neq j}\int_0^T\int_{\R^{3N}}f_N^\eps(t) |v^i-v^j|^{\gamma+2}I_{|v^i-v^j|\leq\frac{1}{\eps}}\Pi(|v^i-v^j|)\\\&\qquad\qquad\qquad\qquad\qquad\qquad:\left((\nabla_{v^i}-\nabla_{v^j})\log f_N^\eps(t)\right)^{\otimes2}\d V_N\d t\\
\&-\frac{\eps}{N}\sum_{i=1}^N\int_0^T\int_{\R^{3N}}\left|\nabla_{v^i}\log f_N^\eps(t)\right|^2f_N^\eps\d V_N\d t.
\end{aligned}
$$
Then, the entropy is uniformly-in-$\eps$ bounded such as
$$
\frac{1}{N}\int_0^T\int_{\R^{3N}}f_N^\eps(t)\log f_N^\eps(t)\d V_N \d t<\infty.
$$
We then test \eqref{eq vanishing} against $|V_N|^2=\sum_{i=1}^N|v^i|^2$ and get
$$
\begin{aligned}
\int_{\R^{3N}}f_N^\eps(T) |V_N|^2\d V_N=\&N\int_{\R^{3}}f^{\eps,0}(v^1)|v^1|^2\d v^1+\eps\sum_{i=1}^N\int_0^T\int_{\R^{3N}}f_N^\eps(t)\Delta_{v^i}|V_N|^2\d V_N\d t\\=\&N\int_{\R^{3}}f^{\eps,0}(v^1)|v^1|^2\d v^1+6N\eps T.
\end{aligned}
$$
Then the second moment of $f_{N}^\eps$ on $[0,T]\times\R^{3m}$ is uniformly-in-$\eps$ bounded.  By Dunford-Pettis criterion, for fixed $N\in\N$, there exist a function $f_N\in L^\infty(0,T; L^1(\R^{3N}))$ with finite entropy and finite second moment  such that, there exists a subsequence $\eps'\to0$ such that
$$
f_N^{\eps'}\rightharpoonup f_N\quad\text{weakly in}\quad L^1([0,T]\times\R^{3N}),
$$
namely, for any $\psi_N\in L^\infty([0,T]\times\R^{3N})$, 
$$
\int_{\R^{3N}}f_N^\eps(t)\psi_N\d V_N\rightarrow \int_{\R^{3N}}f_N(t)\psi_N\d V_N \quad \text{as} \quad\eps'\to0.
$$
We deduce that for any $\vphi_N\in C_c^2(\R^{3N})$, as $\eps'\to0$, it holds
$$
\begin{aligned}
\&\bigg|\int_0^T\int_{\R^{3N}}A(v^i-v^j):(\nabla_{v^iv^i}^2\vphi_{N}-\nabla_{v^iv^j}^2\vphi_{N})f_N(t)\d V_{N}\d t\\
\&\qquad-\int_0^T\int_{\R^{3N}}A(v^i-v^j):(\nabla_{v^iv^i}^2\vphi_{N}-\nabla_{v^iv^j}^2\vphi_{N})f_N^{\eps'}(t)\d V_{N}\d t\bigg|\to 0.
\end{aligned}
$$
Similarly, for any $\vphi_N\in C_c^2(\R^{3N})$, as $\eps'\to0$, it holds
$$
\begin{aligned}
\&\bigg|\int_0^T\int_{\R^{3N}}B(v^i-v^j)\cdot(\nabla_{v^i}\vphi_{N}-\nabla_{v^j}\vphi_{N})f_N(t)\d V_{N}\d t\\
\&\qquad-\int_0^T\int_{\R^{3N}}B(v^i-v^j)\cdot(\nabla_{v^i}\vphi_{N}-\nabla_{v^j}\vphi_{N})f_N^{\eps'}(t)\d V_{N}\d t\bigg|\to 0,
\end{aligned}
$$
and 
$$
\eps'\left|\int_0^T\int_{\R^{3N}}f_N^{\eps'}(r)\Delta_{v^i}\vphi_N\d V_N\d t\right|\to0.
$$
In conclusion, $f_N$ satisfies \eqref{eq weak Liouville} and solves the Liouville equation \eqref{eq Liouville} in weak form. The uniqueness of weak solutions can be verified by the relative entropy, for instance, see \cite[Remark 2.6]{carrillo2025fisher}.
\end{proof}

By integrating the unique weak solution of the Liouville equation \eqref{eq Liouville} with respect to the variables $v^{m+1},\ldots,v^N$, we define the marginal $f_{N,m}$ such as \eqref{def marginal}. It can then be verified that $f_{N,m}$ is the unique weak solution of the corresponding BBGKY hierarchy \eqref{eq BBGKY}, whose weak solutions are defined as below. We denote $V_m=(v^1,\ldots,v^m)\in \R^{3m}$.
\begin{definition}[Weak solutions of the BBGKY hierarchy]\label{def BBGKY hierarchy}
For each fixed $m\in\N$ with $1\leq m\leq N-1$ and for any time $T>0$, a weak solution of the BBGKY hierarchy \eqref{eq BBGKY} with $f_{N,m}(0)=(f^0)^{\otimes m}$ satisfies the following
weak form of hierarchy, for any  $\vphi_{m}\in C^{2}_c(\R^{3m})$
\begin{align}\label{eq weak BBGKY}
\notag\int_{\R^{3m}}f_{N,m}(T)\&\vphi_{m} \d V_{m}-\int_{\R^{3m}}(f^0)^{\otimes m}\vphi_{m}\d V_{m}\\\notag=\&\frac{1}{N}\sum_{i\neq j}^m\int_0^T\int_{\R^{3m}}A(v^i-v^j):(\nabla_{v^iv^i}^2\vphi_{m}-\nabla_{v^iv^j}^2\vphi_{m})f_{N,m}\d V_{m}\d t\\\&+\frac{1}{N}\sum_{i\neq j}^m\int_0^T\int_{\R^{3m}}B(v^i-v^j)\cdot (\nabla_{v^i}\vphi_{m}-\nabla_{v^j}\vphi_{m}) f_{N,m}\d V_{m}\d t\\\notag
\&+ \frac{N-m}{N}\sum_{i=1}^m \int_0^T\int_{\R^{3(m+1)}}  A(v^i-v^{m+1}):\nabla_{v^iv^i}^2\vphi_{m}\,f_{N,m+1}\d V_{m+1}\d t\\\notag
\&+ \frac{2(N-m)}{N}\sum_{i=1}^m \int_0^T\int_{\R^{3(m+1)}}  B(v^i-v^{m+1})\cdot \nabla_{v^{i}}\vphi_{m}\,f_{N,m+1}\,\d V_{m+1}\d t.
\end{align}
Moreover, $f_{N,m}$ represents the density of the joint distribution $\law(V^1,\ldots,V^m)$ of the conservative particles \eqref{kacparticles}, and it has the finite second moment and finite entropy. 
\end{definition}

We remain to show that the weak solution of the BBGKY hierarchy converges to a weak solution of the infinite Landau hierarchy defined below.

\begin{definition}[Weak solutions of the infinite Landau hierarchy]\label{def Landau hierarchy}
For each fixed $m\in\N$ and for any $T>0$, a weak solution of Landau hierarchy \eqref{eq Landau hierarchy} with $f_m=(f^0)^{\otimes m}$ satisfies the following
weak form, for any $\vphi_{m}\in C^2_b(\R^{3m})$, 
\begin{equation}\label{eq weak Landau hierarchy}
\begin{aligned}
  \int_{\R^{3m}}\&f_{m}(T)\vphi_{m}\d V_{m}-\int_{\R^{3m}}(f^0)^{\otimes m}\vphi_{m}\d V_{m}\\=\&\sum_{i=1}^m \int_0^T\int_{\R^{3(m+1)}}  A(v^i-v^{m+1}):\nabla_{v^iv^i}^2\vphi_{m}\,f_{m+1}\d V_{m+1}\d t\\
\&+2\sum_{i=1}^m\int_0^T\int_{\R^{3(m+1)}}  B(v^i-v^{m+1})\cdot \nabla_{v^{i}}\vphi_{m}\,f_{m+1}\d V_{m+1}\d t.
\end{aligned}
\end{equation} 
Moreover, $f_{m}$ has finite entropy and satisfies the exponential moment bound
\begin{equation}\label{ineq fm exponential}
\sup_{t\in[0,\infty)}\int_{\R^{3m}}\exp(\xi_{r_0,\gamma}|V_m|^\frac{4}{\gamma+2})f_m(t)\d V_m<\infty.
\end{equation}
\end{definition}

\begin{proof}[Proof of Theorem \ref{thm existence}]
By the sub-additivity of the entropy, the $m$-th marginal $f_{N,m}$ satisfies the following uniform-in-$N$ estimate such as
$$
\begin{aligned}
\sup_{t\in[0,\infty)}\int_{\R^{3m}} f_{N,m}(t)\log f_{N,m}(t)\d V_m\leq \&\sup_{t\in[0,\infty)}\frac{m}{N}\int_{\R^{3N}} f_{N}(t)\log f_{N}(t)\d V_N\\\leq \&m\int_{\R^{3}}f^0(v^1)\log f^0(v^1)\d v^1,
\end{aligned}
$$
and 
$$
\sup_{t\in[0,\infty)}\int_{\R^{3m}} f_{N,m}(t)(1+|V_m|^2)\d V_m=1+m\int_{\R^3}f^0(v^1)|v^1|^2\d v^1.  $$
By Dunford-Pettis criterion, there exist a function $f_m\in L^\infty(0,T;L^1(\R^{3m}))$ with finite entropy and finite second moment  such that for a subsequence $N'\to\infty$, it holds 
$$
f_{N',m}\rightharpoonup f_m\quad\text{weakly in}\quad L^1([0,T]\times\R^{3m}).
$$
Namely, for any $\psi_m\in L^\infty([0,T]\times\R^{3m})$, it has
\begin{equation}\label{conv L1}
\int_0^T\int_{\R^{3m}}f_{N,m}\psi_m\d V_m\d t\rightarrow \int_0^T\int_{\R^{3m}}f_m\psi_m\d V_m\d t\quad \quad \text{as} \quad N\to0.
\end{equation}
Since $f_{N,m}$ is uniformly-in-$N$ bounded in $L^\infty(0,T;L^1(\R^{3m}))$ then, it yields
$$
f_{N',m}\to f_m\quad\text{in}\quad L^1(0,T;\P(\R^{3m})),
$$
where the weakly convergence in probability space is understood as
$
\W_2(f_{N,m},f_m)\to 0
$ as $N\to \infty$.

It remains to pass to the limit. For the first two terms on the right-hand side of \eqref{eq weak BBGKY}, the corresponding sum vanishes when $N\to\infty$. Since for any $\vphi_{m}\in C^2_b(\R^{3m})$, it can be bounded such as
$$
\begin{aligned}
\&\bigg|\frac{1}{N}\sum_{i\neq j}^m\int_0^T\int_{\R^{3m}}A(v^i-v^j):(\nabla_{v^iv^i}^2\vphi_{m}-\nabla_{v^iv^j}^2\vphi_{m})f_{N,m}\d V_{m}\d t\\\&\qquad+\frac{1}{N}\sum_{i\neq j}^m\int_0^T\int_{\R^{3m}}B(v^i-v^j)\cdot (\nabla_{v^i}\vphi_{m}-\nabla_{v^j}\vphi_{m}) f_{N,m}\d V_{m}\d t\bigg|\\
\leq \& \frac{m^2}{N}\bigg(\left\|A(v^1-v^2):(\nabla_{v^1v^1}^2\vphi_{m}-\nabla_{v^1v^2}^2\vphi_{m})\right\|_{L^\infty}\\\&\qquad+\left\|B(v^1-v^2)\cdot (\nabla_{v^1}\vphi_{m}-\nabla_{v^2}\vphi_{m})\right\|_{L^\infty}\bigg)\int_0^T\int_{\R^{3m}}f_{N,m}\d V_m
\\\leq\& \frac{m^2TC(\vphi_m)}{N}.
\end{aligned}
$$
For the third term on the right-hand side, we have the following estimate: for any  $\vphi_m\in C_c^2(\R^{3N})$, it holds
$$
\begin{aligned}
\bigg|\sum_{i=1}^m\&\int_0^T\int_{\R^{3(m+1)}}  A(v^i-v^{m+1}):\nabla_{v^iv^i}^2\vphi_{m}\,f_{m+1}\d V_{m+1}\d t\\
\&\qquad-\frac{N-m}{N}\sum_{i=1}^m\int_0^T\int_{\R^{3(m+1)}}  A(v^i-v^{m+1}):\nabla_{v^iv^i}^2\vphi_{m}\,f_{m+1,N}\d V_{m+1}\d t\bigg|\\
\leq\&m\bigg|\int_0^T\int_{\R^{3(m+1)}}  A(v^1-v^{m+1}):\nabla_{v^1v^1}^2\vphi_{m}\left(f_{m+1}(t)-f_{m+1,N}(t)\right)\d V_m\d t\bigg|\\
\&+\frac{m^2}{N}\bigg|\int_0^T\int_{\R^{3(m+1)}}  A(v^1-v^{m+1}):\nabla_{v^1v^1}^2\vphi_{m}f_{m+1,N}(t)\d V_m\d t\bigg|.
\end{aligned}
$$
The right-hand side converges to zero as subsequence $N\to0$, by the weak convergence result \eqref{conv L1}. Similarly, it holds
$$
\begin{aligned}
\bigg|\sum_{i=1}^m\&\int_0^T\int_{\R^{3(m+1)}}  B(v^i-v^{m+1})\cdot\nabla_{v^i}\vphi_{m}\,f_{m+1}\d V_{m+1}\d t\\
\&\qquad-\frac{N-m}{N}\sum_{i=1}^m\int_0^T\int_{\R^{3(m+1)}}  B(v^i-v^{m+1})\cdot\nabla_{v^i}\vphi_{m}\,f_{m+1,N}\d V_{m+1}\d t\bigg|\\
\leq\&m\bigg|\int_0^T\int_{\R^{3(m+1)}}  B(v^1-v^{m+1})\cdot\nabla_{v^1}\vphi_{m}\left(f_{m+1}(t)-f_{m+1,N}(t)\right)\d V_m\d t\bigg|\\
\&+\frac{m^2}{N}\bigg|\int_0^T\int_{\R^{3(m+1)}}  B(v^1-v^{m+1})\cdot\nabla_{v^i}\vphi_{m}f_{m+1,N}(t)\d V_m\d t\bigg|,
\end{aligned}
$$
which converges to zero as $N\to \infty$.

Therefore, for any $T>0$, $f_m\in L^\infty(0,T;L^1(\R^{3m}))$ solves \eqref{eq weak Landau hierarchy} in the weak sense. And the estimate \eqref{ineq fm exponential} can be proved in the same spirit as Corollary \ref{coro exponential} below.

\end{proof}

\begin{corollary}[Moment estimate of the hierarchy]\label{coro exponential}
Under Assumption \ref{assumption hard kac}, there exists some constant $\xi_{r_0,\gamma}$ depending only on $r_0$ and $\gamma$ such that. The weak solutions of the first equation \eqref{eq Landau hierarchy} obtained in Theorem \ref{thm existence} satisfies the following uniform-in-time exponential moment bound
$$
\sup_{t\in[0,\infty)}\int_{\R^{3}}\exp(\xi_{r_0,\gamma}|v|^\frac{4}{\gamma+2})f_1(t,v)\d v<\infty.
$$      
\end{corollary}
\begin{remark}
Notice that the order fulfils $\frac{4}{\gamma+2}\geq \frac{4}{3}$ as $\gamma\in(0,1]$. It is crucial to allow the order greater than $1$ to show the uniqueness of weak solutions of the Landau hierarchy. 
\end{remark}

We conclude this section by proving Corollary \ref{coro exponential}.
\begin{proof}[Proof of Corollary \ref{coro exponential}]
Notice that the estimate \eqref{ineq exponential} gives us the uniform-in-$N$ and uniform-in-time boundedness. We remain to prove the equi-continuity. Recall the weak form of BBGKY hierarchy \eqref{eq weak BBGKY} when $m=1$
$$
\begin{aligned}
\int_{\R^{3}}\vphi \&f_{N,1}(t)\d v^1-\int_{\R^{3}}\vphi f_{N,1}(s)\d v^1\\=\& \frac{N-1}{N} \int_s^t\int_{\R^{6}}  A(v^1-v^{2}):\nabla_{v^1v^1}^2\vphi\,f_{N,2}(r,v^1,v^2)\d v^1\d v^2\d r\\
\&+ \frac{2(N-1)}{N} \int_s^t\int_{\R^{6}}  B(v^1-v^2)\cdot \nabla_{v^{1}}\vphi\,f_{N,2}(r,v^1,v^2)\d v^1\d v^2\d r.
\end{aligned}
$$
We substitute $\vphi$ by $\exp(\xi_{r_0,\gamma}|v|^\frac{4}{\gamma+2})$ and obtain
$$
\left|\int_{\R^{3}}\exp(\xi_{r_0,\gamma}|v|^\frac{4}{\gamma+2}) f_{N,1}(t,v)\d v-\int_{\R^{3}}\exp(\xi_{r_0,\gamma}|v|^\frac{4}{\gamma+2}) f_{N,1}(s,v)\d v\right|\leq C|t-s|,
$$
where the constant $C$ on the right-hand side is independent of $s,t$ and $N$ because the marginal can propagate the exponential moment with the parameter $\xi$ sightly greater than $\xi_{r_0,\gamma}$. Then, by Arzelà–Ascoli theorem, there exists subsequence such that 
$$
\int_{\R^{3}}\exp(\xi_{r_0,\gamma}|v^1|^\frac{4}{2+\gamma})f_{N',1}(t,v^1)\d v^1\to \int_{\R^{3}}\exp(\xi_{r_0,\gamma}|v^1|^\frac{4}{2+\gamma})f_{1}(t,v^1)\d v^1<\infty,
$$
where the convergence is uniform-in-time.
\end{proof}
\section{Uniqueness of solutions of infinite Landau hierarchy}\label{sec uniqueness}

In this section, we establish the stability estimate for the infinite Landau hierarchy \eqref{eq Landau hierarchy}, as presented in Theorem \ref{thm uniqueness}. This estimate implies the uniqueness of weak solutions of the Landau hierarchy, as an important step toward proving propagation of chaos.

To derive this estimate, we construct two families of SDEs whose distributions solve the Landau hierarchy \eqref{eq Landau hierarchy} up to $n$-th level, starting from different initial data. By coupling the first $m$ components (with $m\leq n$) from each family, we obtain an $L^2$-estimate on trajectories in $\R^{3m}$ denoted by $u_m^{(n)}$, which controls the distance of two solutions of the Landau hierarchy. We emphasise the dependence of this quantity on $n$ constructed SDEs, which fulfils the consistency condition in Remark \ref{rmk consistency}. The quantity $u_m^{(n)}$ satisfies an inequality hierarchy involving $u_{m+1}^{(n)}$, as stated in Lemma \ref{lemma inequality hierarchy}. Using the exponential moment estimate \eqref{ineq fm exponential} and a combinatorial argument from \cite{lacker2023hierarchies}, we will obtain an estimate on $u_m^{(n)}$ (see Lemma \ref{lemma umn}) that is independent of higher-order terms in the hierarchy. 

\medskip

We begin by formulating suitable SDEs corresponding the Landau hierarchy \eqref{eq Landau hierarchy}:
\begin{itemize}
\item We fix some large $n\in\N$ and truncate the infinite hierarchy at $n$-th level.
\item Given two weak solutions $f_1$ and $\t f_1$ of the first equation of the infinite Landau hierarchy \eqref{eq Landau hierarchy}, with initial data $f^0$ and $\t f^0$ respectively, we construct two stochastic processes $(U^{n+1}_t)_{t\geq 0}$ and $(\t U^{n+1}_t)_{t\geq 0}$ on some filtered probability space such that $\law(U^{n+1})=f_1$, $\law(\t U^{n+1})=\t f_1$ and they form an optimal coupling in the sense that $$
\E\left[|U_t^{n+1}-\t U_t^{n+1}|^2\right]=\W_2^2(f_1(t),\t f_1(t)),\quad\text{for any}\,\, t\geq0.
$$
\item We construct $n$ independent copies of the optimal coupling pairs $(U_0^i,\t U_0^i)$, $i=1,2,\ldots,n$ such that $\law(U_0^i)=f^0$ and $\law(\t U_0^i)=\t f^0$, and denote
$$
u_0:=\W_2^2\left(f^0,\t f^0\right)=\E\left[|U_0^i-\t U_0^i|^2\right],\quad i=1,\ldots,n.
$$
We then define two collections of initial data: $\mathbf{U}_0=(U^1_0,\ldots, U^n_0)$ and $\mathbf{\t U}_0=(\t U^1_0,\ldots, \t U^n_0)$ so that $\law(\mathbf{U}_0)=(f^0)^{\otimes m}$ and $\law(\mathbf{U}_0)=(\t f^0)^{\otimes m}$.
\item  Given two sequences of weak solutions $(f_1,f_2,f_3,\ldots,f_n)$ and $(\t f_1,\t f_2,\t f_3,\ldots,\t f_n)$, the prepared initial data $\mathbf{U}_0=(U^1_0,\ldots, U^n_0)$ and $\mathbf{\t U}_0=(\t U^1_0,\ldots, \t U^n_0)$, and the constructed processes $(U^{n+1}_t)_{t\geq 0}$ and $(\t U^{n+1}_t)_{t\geq 0}$, we can define two systems of SDEs $\mathbf{U}=(U^1,U^2,\ldots,U^n)$ and $\mathbf{\t U}=(\t U^1,\t U^2,\ldots,\t U^n)$ which solve equations \eqref{eq SDE hierarchy} and \eqref{eq SDE hierarchy2}, respectively. The joint distributions within each family satisfy   $$\law(U^1,U^2,\ldots,U^n)=f_n,\quad \law(U^1,U^2,\ldots,U^n,U^{n+1})=f_{n+1},$$ $$\law(\t U^1,\t U^2,\ldots,\t U^n)=\t f_n,\quad \law(\t U^1,\t U^2,\ldots,\t U^n,\t U^{n+1})=\t f_{n+1}.$$
The solutions of these SDEs are understood in the sense of martingale problems; see Remark \ref{rmk martingale}.

\item There exists a common filtered probability space equipped with a shared Brownian motion  $\mathbf{W}=(W^1,W^2,\ldots,W^n)$ 
such that $(\Omega,\F,(\F_t)_{t\geq0},\mathbb{P}, \mathbf{U}, \mathbf{W})$ and $(\Omega,\F,(\F_t)_{t\geq0},\mathbb{P}, \mathbf{\t U}, \mathbf{W})$ are weak solutions of \eqref{eq SDE hierarchy} and \eqref{eq SDE hierarchy2}, respectively; see for example  \cite[Proposition 3.6]{li2019quantitative}.
\end{itemize}

Recall the vector $B$ in \eqref{def AB} and the diffusion matrix $\sigma$ 
$$
B(z)=-2z|z|^\gamma,\quad \sigma(z)=|z|^{1+\frac{\gamma}{2}}\Pi(z).
$$
The the SDE for $\mathbf{U}=(U^1,U^2,\ldots,U^n)\in\R^{3n}$ takes the form: 
\begin{align}\label{eq SDE hierarchy}
\&\displaystyle\d\left(\begin{array}{c}
     U_t^1  \\U_t^2  \\
    \vdots  \\U_t^{n-1}  \\
    U_t^n 
\end{array}\right)= 2\left(\begin{array}{c}
     B(U_t^1-U_t^{n+1})  \\B(U_t^2-U_t^{n+1})  \\
    \vdots  \\B(U_t^{n-1}-U_t^{n+1})  \\
    B(U_t^n-U_t^{n+1})
\end{array}\right)\d t\\\&\quad+\sqrt{2}\left(\begin{array}{ccccc}
     \sigma(U_t^1-U_t^{n+1})& 0& \cdots  & 0 & 0   \\
     0& \sigma(U_t^2-U_t^{n+1})& \cdots  & 0 & 0   \\
    \vdots& 0 & \ddots &  0 & \vdots \\
  0 &0&\cdots &\sigma(U_t^{n-1}-U_t^{n+1})& 0\\
  0 &0&\cdots &0& \sigma(U_t^n-U_t^{n+1})
\end{array}\right)\left(\begin{array}{c}
     \d W_t^1  \\\d W_t^2  \\
    \vdots  \\\d W_t^{n-1}  \\
    \d W_t^n 
\end{array}\right)\notag,
\end{align}
with the initial data $\mathbf{U_0}=(U^1_0,U^2_0,\ldots,U^n_0)$, and the SDE for $\mathbf{\t U}=(\t U^1,\t U^2,\ldots,\t U^n)$ takes the form:
\begin{align}\label{eq SDE hierarchy2}
\&\displaystyle\d\left(\begin{array}{c}
    \t U_t^1  \\\t U_t^2  \\
    \vdots  \\\t U_t^{n-1}  \\
    \t U_t^n 
\end{array}\right)= 2\left(\begin{array}{c}
     B(\t U_t^1-\t U_t^{n+1})  \\B(\t U_t^2-\t U_t^{n+1})  \\
    \vdots  \\B(\t U_t^{n-1}-\t U_t^{n+1})  \\
    B(\t U_t^n-\t U_t^{n+1})
\end{array}\right)\d t\\\&\quad+\sqrt{2}\left(\begin{array}{ccccc}
     \sigma(\t U_t^1-\t U_t^{n+1})& 0& \cdots  & 0 & 0   \\
     0& \sigma(\t U_t^2-\t U_t^{n+1})& \cdots  & 0 & 0   \\
    \vdots& 0 & \ddots &  0 & \vdots \\
  0 &0&\cdots &\sigma(\t U_t^{n-1}-\t U_t^{n+1})& 0\\
  0 &0&\cdots &0& \sigma(\t U_t^n-\t U_t^{n+1})
\end{array}\right)\left(\begin{array}{c}
     \d W_t^1  \\\d W_t^2  \\
    \vdots  \\\d W_t^{n-1}  \\
    \d W_t^n 
\end{array}\right)\notag.
\end{align}
One can verify by using It\^{o}'s formula that the SDE hierarchy formally corresponds to the Landau hierarchy \eqref{eq Landau hierarchy}:
$$
\p_t f_{n}=\sum_{i=1}^n\div_{v^i}\left[\int_{\R^3}\Big(A(v^i-v^{n+1})\nabla_{v^i}f_{n+1}-B(v^i-v^{n+1})f_{n+1}\Big)\d v^{n+1}\right].
$$
\begin{remark}\label{rmk martingale}
We understand the existence of solutions of the SDE hierarchy \eqref{eq SDE hierarchy} by the martingale problem. For example, for $\vphi\in C_b^2(\R^3)$, we define
$$
\begin{aligned}
\M_t^1:=\&\vphi(U_t^1)-\vphi(U_0^1)\\\&\qquad-2\int_0^t\nabla\vphi(U_s^1)\cdot B(U^1_s-U^{n+1}_s)\d s-\int_0^tA(U^1_s-U^{n+1}_s):\nabla^2\vphi(U^1_s)\d s,
\end{aligned}
$$
where for all $t\geq0$ it holds $\law(U_t^1)=f_1(t,v^1)$ and $\law(U_t^1,U_t^{n+1})=f_2(t,v^1,v^{n+1})$.
To show $\M_t$ is a square-integrable martingale, we take expectation to get
$$
\begin{aligned}
\E\left[\M^1_t\right]=\&\int_{\R^3}\vphi(v^1)f_1(t,v^1)\d v^1-\int_{\R^3}\vphi(v^1)f^0(v^1)\d v^1\\\&-2\int_0^t\int_{\R^6}\nabla\vphi(v^1)\cdot B(v^1-v^{n+1})f_2(s,v^1,v^{n+1})\d s\\\&+\int_0^t\int_{\R^6}A(v^1-v^{n+1}):\nabla^2\vphi(v_1)f_2(s,v^1,v^{n+1})\d s.
\end{aligned}
$$
Since Definition \ref{def Landau hierarchy} (the weak formula \eqref{eq weak Landau hierarchy}) and the existence result Proposition \ref{thm existence}, the right-hand side equals to $0$ and we obtain 
$$
\E[\M^1_t]=0,\quad\forall t\geq0.
$$
Also, applying the moment estimate \eqref{ineq fm exponential}, we have
$$
\begin{aligned}
\E[|\M^1_t|^2]=\&\E\left[\int_0^t\nabla\vphi(U^1_s)\sigma(U^1_s-U^{n+1}_s)\sigma^T(U^1_s-U^{n+1}_s)(\nabla\vphi(U^1_s))^T\d s\right]
\\=\&\int_0^t\int_{\R^{6}}A(v^1-v^{n+1}):(\nabla\vphi(v^1))^{\otimes2}f_2(s,v^1,v^{m+1})\d s\d v^1\d v^{m+1}\\
\leq\&\|\nabla\vphi\|^2_{L^\infty}\int_0^t\int_{\R^{6}}|v^1-v^{m+1}|^{2+\gamma}f_2(s,v^1,v^{m+1})\d s\d v^1\d v^{m+1}<\infty.  
\end{aligned}
$$
Then $\M^1$ is a square-integrable martingale, and there exists a suitable filtered probability space and a 3-dimension Brownian motion $(W^1_t)_{t\geq0}$ such that 
$$
\d \M^1_t=\nabla\vphi(U_t^1)\sigma(U_t^1-U_t^{n+1})\d W_t^1.
$$
We deduce that there exists a weak solution $(\Omega,(\F_t)_{t\geq 0}, \mathbb{P}, U^1,W^1)$ of the SDE
$$
\d U^1_t= 2B(U^1_t-U^{n+1}_t)\d t+\sqrt{2}\sigma(U^1_t-U^{n+1}_t)\d W_t^1.
$$
In the spirit, we can show that there exits a filtered probability space $(\Omega,\F,(\F_t)_{t\geq 0}, \mathbb{P}, \mathbf{U},\mathbf{W})$ with $\mathbf{U}=(U_1,U_2,\ldots,U_n)$  and Brownian motion  $\mathbf{W}=(W_1,W_2,\ldots,W_n)$ satisfied the SDE \eqref{eq SDE hierarchy}.  
\end{remark}

After the preparation, we can compute the difference between the first $\ell$ ($\ell\leq n$) components in two families, namely $(U^1,U^2,\ldots,U^\ell)$  and $(\t U^1,\t U^2,\ldots,\t U^\ell)$, thanks to the same Brownian. The difference SDE on $\R^{3\ell}$ reads
$$
\begin{aligned}
\displaystyle\&\d\left(\begin{array}{c}
     U_t^1-\t U_t^1   \\
    \vdots    \\
    U_t^\ell-\t U_t^\ell 
\end{array}\right)= 2\left(\begin{array}{c}
     B(U_t^1-U_t^{n+1})-B(\t U_t^1-\t U_t^{n+1})    \\
    \vdots    \\
    B( U_t^\ell- U_t^{n+1})-B(\t U_t^\ell-\t U_t^{n+1})
\end{array}\right)\d t\\\&\quad+\sqrt{2}\left(\begin{array}{ccc}
     \sigma(U_t^1-U_t^{n+1})-\sigma(\t U_t^1-\t U_t^{n+1})\,\,& \cdots  & 0   \\
    \vdots& \ddots  & \vdots \\
  0 &\cdots & \,\,\sigma(U_t^\ell-U_t^{n+1})-\sigma(\t U_t^\ell-\t U_t^{n+1})
\end{array}\right)\left(\begin{array}{c}
     \d W_t^1   \\
    \vdots    \\
    \d W_t^\ell 
\end{array}\right).
\end{aligned}
$$
For $\ell\leq n$, we denote the $L^2$-distance by
$$
u_\ell^{(n)}(t):=\sum_{i=1}^\ell\E\left[\left|U_t^i-\t U_t^i\right|^2\right].
$$

\begin{remark}[Consistency]\label{rmk consistency} For each different $n$, we will construct different families of SDEs, which are not necessarily the same. However, for any fixed $m$, the first $m$ SDEs in different families share the same distribution $\law(U_t^1,\ldots,U_t^m)=f_m(t)$. The distance of two weak solutions of the Landau hierarchy can be bounded such as for any $n$, 
\begin{equation}\label{ineq consistency3}
 \W_2^2\left(f_m(t),\t f_m(t)\right)\leq u_m^{(n)} (t), \quad t\geq 0.  
\end{equation}
Under Assumption \ref{assumption hard kac}, the constructed $L^2$-distance satisfies the direct bound for any $1\leq \ell \leq n$, 
\begin{equation}\label{ineq trivial bound}
\begin{aligned}
\sup_{t\in[0,\infty)}  u_\ell^{(n)}(t)=\&\sup_{t\in[0,\infty)}\sum_{i=1}^\ell\E\left[\left|U_t^i-\t U_t^i\right|^2\right]\leq \sup_{t\in[0,\infty)}\sum_{i=1}^\ell\E\left[\left|U_t^i|^2+|\t U_t^i\right|^2\right]\\
\leq\&\sup_{t\in[0,\infty)}2\ell\int_{\R^{3}}|v|^2\big(f_1(t,v)+\t f_1(t,v)\big)\d v\leq 2\ell r_0^2.
\end{aligned}
\end{equation} 
Let the truncated $n$ levels hierarchy be fixed. Then the processes $U^{n+1}$ and $\t U^{n+1}$ are constructed to be the optimal coupling. Though the pair $(U_t^{m+1},\t U_t^{m+1})$ and $(U_t^{n+1},\t U_t^{n+1})$ have the same marginal distribution, it has the inequality
\begin{equation}\label{ineq consistency2}
\E\left[|U_t^{n+1}-\t U_t^{n+1}|^2\right]=\W_2^2(f_1(t),\t f_1(t))\leq \E\left[|U_t^{m+1}-\t U_t^{m+1}|^2\right]=u_{m+1}^{(n)}(t)-u_m^{(n)}(t).
\end{equation}
Also, the initial $L^2$-distance satisfies the identity
\begin{equation}\label{ineq consistency1}
u_\ell^{(n)}(0):=\sum_{i=1}^\ell\E\left[\left|U_0^i-\t U_0^i\right|^2\right]=\ell u_0=\ell \W_2(f^0,\t f^0).
\end{equation}
\end{remark}

Let the truncated $n$-level hierarchy be fixed. The constructed $L^2$-distance satisfies the following ordinary differential inequality.
\begin{lemma}\label{lemma inequality hierarchy}
For any fixed $m\in \N$, the following inequality holds, for 
$t_m\in[0,T]$
\begin{equation}\label{ineq simplified}
\begin{aligned}
u_m^{(n)}(t_m)
\leq u_m^{(n)}(0)+a\int_0^{t_m}\&\left[mu^{(n)}_{m+1}(t_{m+1})-(m-1)u^{(n)}_m(t_{m+1})\right]\d t_{m+1}\\ \&+\frac{C_2Tm}{\exp [C_1a^{4/(\gamma^2+2\gamma)}]},
\end{aligned}    
\end{equation}
where $a$ is a cut-off independent of $m,n,t$ which will be chosen later and the constants $C_1$ and $C_2$ are independent of $m$, $T$ and $a$. 
\end{lemma}
\begin{proof}[Proof of Lemma \ref{lemma inequality hierarchy}]

For any $T\geq0$, It\^{o}'s formula implies
\begin{equation}\label{eq ito2}
\begin{aligned}
\sum_{i=1}^m\E\left[\left|U_T^i-\t U_T^i\right|^2\right]\& =\sum_{i=1}^m\E\left[\left|U_0^i-\t U_0^i\right|^2\right]\\\&+2\sum_{i=1}^m\int_0^T\E\left[\left(U_t^i-\t U_t^i\right)\cdot \left(B(U_t^i-U_t^{n+1})-B(\t U_t^i-\t U_t^{n+1})\right)\right] \d t
\\\&+\sum_{i=1}^m\int_0^T\E\left[\left\|\sigma(U_t^i-U_t^{n+1})-\sigma(\t U_t^i-\t U_t^{n+1})\right\|^2\right] \d t.
\end{aligned}
\end{equation}
Recall the following lemma, for example \cite[Lemma 8]{fournier2017kac},
\begin{lemma}\label{ineq b sigma} 
For any $\gamma\in(0,1]$, and any $x,y\in\R^3$
$$
|B(x)-B(y)|\leq C_B|x-y|(|x|^\gamma+|y|^\gamma),
$$
and 
$$
\|\sigma(x)-\sigma(y)\|\leq C_\sigma|x-y|(|x|^{\gamma/2}+|y|^{\gamma/2}).
$$
\end{lemma}
We get the estimate of \eqref{eq ito2} for some uniform-in-time constant $C_0$ and a cut-off radius $a>0$ such as
\begin{equation*}\label{ineq ito}
\begin{aligned}
\sum_{i=1}^m\E\& \left[\left|U_T^i-\t U_T^i\right|^2\right]-\sum_{i=1}^m\E\left[\left|U_0^i-\t U_0^i\right|^2\right]\\\leq\&2C_B\sum_{i=1}^m\int_0^T\E\bigg[\left(\left|U_t^i-\t U_t^i\right|^2+\left|U_t^i-\t U_t^i\right|\left|U_t^{n+1}-\t U_t^{n+1}\right|\right)\\\&\qquad\qquad\qquad\qquad\qquad\qquad\qquad\times\left(|U_t^i-U_t^{n+1}|^\gamma+|\t U_t^i-\t U_t^{n+1}|^\gamma\right)\bigg] \d t
\\\&+4C_\sigma^2\sum_{i=1}^m\int_0^T\E\bigg[\left(\left|U_t^i-\t U_t^i\right|^2+\left|U_t^{n+1}-\t U_t^{n+1}\right|^2\right)\\\&\qquad\qquad\qquad\qquad\qquad\qquad\qquad\times\left(|U_t^i-U_t^{n+1}|^\gamma+|\t U_t^i-\t U_t^{n+1}|^\gamma\right)\bigg] \d t
\\\leq\&C_0\sum_{i=1}^m\int_0^T\E\bigg[\left(\left|U_t^i-\t U_t^i\right|^2+\left|U_t^{n+1}-\t U_t^{n+1}\right|^2\right)\\\&\qquad\qquad\qquad\qquad\qquad\times|U_t^i-U_t^{n+1}|^\gamma\left(I_{|U_t^i-U_t^{n+1}|< \left(\frac{a}{2C_0}\right)^{1/\gamma}}+I_{|U_t^i-U_t^{n+1}|\geq \left(\frac{a}{2C_0}\right)^{1/\gamma}}\right)\bigg] \d t
\\\&+C_0\sum_{i=1}^m\int_0^T\E\bigg[\left(\left|U_t^i-\t U_t^i\right|^2+\left|U_t^{n+1}-\t U_t^{n+1}\right|^2\right)\\\&\qquad\qquad\qquad\times|\t U_t^i-\t U_t^{n+1}|^\gamma\left(I_{|\t U_t^i-\t U_t^{n+1}|< \left(\frac{a}{2C_0}\right)^{1/\gamma}}+I_{|\t U_t^i-\t U_t^{n+1}|\geq \left(\frac{a}{2C_0}\right)^{1/\gamma}}\right)\bigg] \d t.
\end{aligned}
\end{equation*}
By using the inequality for some $R>0$
$$
\begin{aligned}
|x-y|^\gamma I_{|x-y|\geq 2R}\leq\& (|x|^\gamma+|y|^\gamma) \left(I_{|x|\geq R}+ I_{|y|\geq R} \right)\\\leq\& (|x|^\gamma+|y|^\gamma) \left(\frac{\exp( \xi |x|^\frac{4}{\gamma+2})}{\exp( \xi R^\frac{4}{\gamma+2})}+ \frac{\exp( \xi |y|^\frac{4}{\gamma+2})}{\exp( \xi R^\frac{4}{\gamma+2})} \right),
\end{aligned}
$$
where we take $R=\frac{1}{2}\left(\frac{a}{2C_0}\right)^{1/\gamma}$ and $0<\xi<\xi_{r_0,\gamma}$ with $\xi_{r_0,\gamma}$ defined in Theorem \ref{thm exponential}, 
we obtain

\begin{equation*}
\begin{aligned}
\sum_{i=1}^m\& \E\left[\left|U_T^i-\t U_T^i\right|^2\right]-\sum_{i=1}^m\E\left[\left|U_0^i-\t U_0^i\right|^2\right]
\\\leq\&a \sum_{i=1}^m\int_0^T\E\bigg[\left|U_t^i-\t U_t^i\right|^2+\left|U_t^{n+1}-\t U_t^{n+1}\right|^2\bigg] \d t
\\\&+\frac{C_0}{\exp [C_1a^{4/(\gamma^2+2\gamma)}]}\sum_{i=1}^m\int_0^T\E\Bigg[\left(\left|U_t^i-\t U_t^i\right|^2+\left|U_t^{n+1}-\t U_t^{n+1}\right|^2\right)\\\&\qquad\qquad\qquad\times\Bigg(\left(|U_t^i|^\gamma+|U_t^{n+1}|^\gamma\right)\left(\exp( \xi  |U_t^i|^\frac{4}{\gamma+2})+ \exp( \xi  |U_t^{n+1}|^\frac{4}{\gamma+2}) \right)\\\&\qquad\qquad\qquad\qquad+\left(|\t U_t^i|^\gamma+|\t U_t^{n+1}|^\gamma\right)\left(\exp( \xi  |\t U_t^i|^\frac{4}{\gamma+2})+ \exp( \xi  |\t U_t^{n+1}|^\frac{4}{\gamma+2}) \right)\Bigg)\Bigg] \d t\\\leq\&a \sum_{i=1}^m\int_0^T\E\bigg[\left|U_t^i-\t U_t^i\right|^2+\left|U_t^{n+1}-\t U_t^{n+1}\right|^2\bigg] \d t+\frac{mTC_2(r_0,\gamma) }{\exp[C_1a^{4/(\gamma^2+2\gamma)}]},
\end{aligned}
\end{equation*}
where the last step is due to $\xi <\xi_{r_0,\gamma}$ and the uniform-in-time propagation of exponential moment for both $f_1$ and $\t f_1$. We estimate the first term on the right-hand side by using \eqref{ineq consistency2}:
$$
\sum_{i=1}^m\E\bigg[\left|U_t^i-\t U_t^i\right|^2+\left|U_t^{n+1}-\t U_t^{n+1}\right|^2\bigg] \leq u_m^{(n)}(t)+m\left(u_{m+1}^{(n)}(t)-u_m^{(n)}(t)\right),
$$   
which concludes the desired inequality.
\end{proof}

We remain to analyse the inequality hierarchy \eqref{ineq simplified}, which will follow the idea in \cite{lacker2023hierarchies}. By the Gronwall's inequality, we deduce that  
\begin{equation}\label{ineq umt}
\begin{aligned}
u_m^{(n)} (t_m)
\leq \& e^{-a(m -1)t_m}u_m^{(n)} (0)\\\&+\int_0^{t_m}e^{-a(m -1)(t_m-t_{m +1})}\left(\frac{C_2mt_m }{\exp[C_1a^{4/(\gamma^2+2\gamma)}]}+am u_{m +1}^{(n)}(t_{m +1})\right)\d t_{m +1}.
\end{aligned}
\end{equation}
Similarly, we evaluate the quantity involving the first $m+1$ SDEs in two constructed families SDEs respectively, namely,
$$
u_{m+1}^{(n)}(t)=\sum_{i=1}^{m+1}\E\left[\left|U_t^i-\t U_t^i\right|^2\right],
$$
and we get the inequality satisfied by $u_{m+1}^{(n)}$
\begin{equation}\label{ineq um+1t}
\begin{aligned}
u_{m+1}^{(n)}\&(t_{m +1})
\leq  e^{-am t_{m +1}}u_{m +1}^{(n)}(0)\\\&+\int_0^{t_{m +1}}e^{-am (t_{m +1}-t_{m +2})}\left(\frac{C_2t_{m+1}(m +1)}{\exp[C_1a^{4/(\gamma^2+2\gamma)}]}+a(m +1)u_{m +2}^{(n)}(t_{m +2})\right)\d t_{m +2}.
\end{aligned}
\end{equation}
For $0\leq t_{m+1}\leq t_m\leq T$, we plug \eqref{ineq um+1t}
into \eqref{ineq umt} to get 
\begin{align}\label{ineq um+1}
\notag u_m^{(n)} \&(t_m )\leq  u_m^{(n)} (0)e^{-a(m -1)t_m }+u_{m +1}^{(n)}(0)am \int_0^{t_m }e^{-a(m -1)(t_m -t_{m +1})-am t_{m +1}}\d t_{m +1}\\\&+\frac{C_2Tam }{a\exp[C_1a^{4/(\gamma^2+2\gamma)}]}\int_0^{t_m }e^{-a(m -1)(t_m -t_{m +1})}\d t_{m +1}\\\&\notag+\frac{C_2Ta^2m (m +1)}{a\exp[C_1a^{4/(\gamma^2+2\gamma)}]}\int_0^{t_m }\int_0^{t_{m +1}}e^{-a(m -1)(t_m -t_{m +1})-am (t_{m +1}-t_{m +2})}\d t_{m +2}\d t_{m +1}\\\notag\&+a^2m (m +1)\int_0^{t_m }\int_0^{t_{m +1}}e^{-a(m -1)(t_m -t_{m +1})-am (t_{m +1}-t_{m +2})}u_{m +2}(t_{m +2})\d t_{m +2}\d t_{m +1}.
\end{align}
By iteration, we can show the following estimate.
\begin{lemma}
Let $0\leq t_{\ell}\leq \cdots\leq t_{m+1}\leq t_m\leq T$, for $m\leq\ell\leq n-1$, we define
$$
\begin{aligned}
\t{F}_m ^\ell(t_m ):=\&\left(\prod_{j=m }^{\ell}aj\right)\int_0^{t_m }\int_0^{t_{m +1}}\cdots\int_0^{t_{\ell}}e^{-a\sum_{j=m }^{\ell}(j-1)(t_j-t_{j+1})}\d t_{\ell+1}\cdots\d t_{m +1}.
\end{aligned}
$$
For $m+1\leq\ell\leq n-1$, we define
$$
\t{G}_m ^\ell(t_m ):=\left(\prod_{j=m }^{\ell-1}aj\right)\int_0^{t_m }\int_0^{t_{m +1}}\cdots\int_0^{t_{\ell-1}}e^{-a\sum_{j=m }^{\ell-1}(j-1)(t_j-t_{j+1})-a(\ell-1)t_{\ell}}\d t_{\ell}\cdots\d t_{m +1},
$$
and $\t G_m ^m (t_m)=e^{-a(m-1)t_m}$.
Then the $L^2$-distance satisfies the estimate as follows:
   $$
\begin{aligned}
u_m^{(n)} (T)\leq\& \sum_{\ell=m }^{n-1}\left[u_\ell^{(n)}(0)\t{G}_m ^\ell(T)+\frac{C_2T}{a\exp[C_1a^{4/(\gamma^2+2\gamma)}]}\t{F}_m ^\ell(T)\right]+ \t{F}_m ^{n-1}(T)\sup_{t\in[0,T]}u_n^{(n)}(t).
\end{aligned}
$$   

\end{lemma}
We further define that for $m\leq\ell\leq n-1$, $F_m ^\ell$ reads as 
$$
F_m ^\ell(t_m ):=\left(\prod_{j=m }^{\ell}aj\right)\int_0^{t_m }\int_0^{t_{m +1}}\cdots\int_0^{t_{\ell}}e^{-a\sum_{j=m }^{\ell}j(t_j-t_{j+1})}\d t_{\ell+1}\cdots\d t_{m +1};
$$
and for $m+1\leq\ell\leq n-1$, $G_m ^\ell$ reads as
$$
G_m ^\ell(t_m ):=\left(\prod_{j=m }^{\ell-1}aj\right)\int_0^{t_m }\int_0^{t_{m +1}}\cdots\int_0^{t_{\ell-1}}e^{-a\ell t_{\ell}-a\sum_{j=m }^{\ell-1}j(t_j-t_{j+1})}\d t_{\ell}\cdots\d t_{m +1},
$$
with $G_m ^m (t_m)=e^{-amt_m}$. Then for $m\leq \ell \leq n-1$, it holds
$$
\t{F}_m ^\ell(t_m )\leq e^{at_m }F_m ^\ell(t_m ),
$$
and
$$
\t G_m ^\ell(t_m )= e^{at_m }G_m ^\ell(t_m ).
$$
Notice that $F_m^\ell$ and $G_m^\ell$ satisfy the estimates below.
\begin{lemma}[{\cite[Lemma 4.8, Proposition 5.1]{lacker2023hierarchies}}]\label{lemma combi}
\begin{equation}\label{ineq F1}
   F_m^{n-1}(t)\leq \exp\left(-2n\left(e^{-aT}-\frac{m}{n}\right)_+^2\right),  
\end{equation}
\begin{equation}\label{ineq F2}
  \sum_{\ell=m}^\infty F_m^\ell(t)\leq m(\exp(at)-1),
\end{equation} 
 and
\begin{equation}\label{ineq G1}\sum_{\ell=m}^\infty\ell G_m^\ell(t)\leq ma\exp(at).  \end{equation} 
\end{lemma}
We refer to the nice argument in \cite[Section 5]{lacker2023hierarchies} for the proof of Lemma \ref{lemma combi}. Next, we have the following estimate.
\begin{lemma}\label{lemma umn}
For $n\geq 2me^{aT}$, it holds
\begin{equation}\label{ineq umT}
u_m^{(n)} (T)\leq u_0 m a\exp(2aT)+\frac{C_2T\exp(2aT)m}{a\exp[C_1a^{4/(\gamma^2+2\gamma)}]}+ \frac{C_4\exp(5aT)}{n},
\end{equation}
where the constants $C_1,C_2$ and $C_4$ are independent of $m,n,T$ and $a$.
\end{lemma}
\begin{proof}[Proof of Lemma \ref{lemma umn}]
By using the estimate \eqref{ineq trivial bound} and \eqref{ineq consistency1}, we have
 $$
\begin{aligned}
u_m^{(n)} (T)\leq\& u_0\sum_{\ell=m }^{n-1}\left[\ell \t{G}_m ^\ell(T)+\frac{C_2T}{a\exp[C_1a^{4/(\gamma^2+2\gamma)}]}\t{F}_m ^\ell(T)\right]+ C_3n \t{F}_m ^{n-1}(T).
\end{aligned}
$$   
We rewrite it in the form with $F_m^\ell$ and $G_m^\ell$ such as
\begin{equation}\label{ineq mn}
\begin{aligned}
u_m^{(n)} (T)\leq\& \exp(aT)u_0\sum_{\ell=m }^{n-1}\left(  \ell G_m ^\ell(T)\right)+\frac{C_2T\exp(aT)}{a\exp[C_1a^{4/(\gamma^2+2\gamma)}]}\sum_{\ell=m }^{n-1} F_m^\ell(T)\\\&\qquad\qquad\qquad+ C_3n\exp(aT) F_m^{n-1}(T).
\end{aligned}
\end{equation}
 We now estimate the right-hand side of \eqref{ineq mn} term by term. For the first term, it holds by \eqref{ineq G1}
$$
u_0\exp(aT)\sum_{\ell=m }^{n-1}\left(  \ell G_m ^\ell(T)\right)\leq u_0 m a\exp(2aT).
$$
The second term has the bound thanks to \eqref{ineq F2}
$$
\frac{C_2T\exp(aT)}{a\exp[C_1a^{4/(\gamma^2+2\gamma)}]}\sum_{\ell=m }^{n-1} F_m ^\ell(T)\leq \frac{C_2T\exp(2aT)m}{a\exp[C_1a^{4/(\gamma^2+2\gamma)}]}.
$$
When $n\geq 2me^{aT}$ and the last can be estimated as
$$
\begin{aligned}
C_3\exp(aT)nF_m^{n-1}(T)\leq \&C_3\exp(aT)n\exp\left(-2n\left(e^{-aT}-\frac{m}{n}\right)_+^2\right)\\\leq \&C_3\exp(aT)n\exp(-\frac{n}{2}e^{-2aT})\\\leq \& C_4 \frac{\exp(5aT)}{n},
\end{aligned}
$$
where the last step thanks to $x^2\exp(-x)\leq 1$ for $x>0$.
Collecting those estimates above, we arrive at the desired bound.
\end{proof}

Finally, we  present the proof of the stability estimate \eqref{ineq stability}.
\begin{proof}[Proof of Theorem \ref{thm uniqueness}]

By using the fact \eqref{ineq consistency3}, it holds for any $n\geq 2m e^{aT}$ that 
$$
\W_2^2\left(f_m(T),\t f_m(T)\right)\leq u_m^{(n)} (T)\leq u_0 m a\exp(2aT)+\frac{C_2T\exp(2aT)m}{a\exp[C_1a^{4/(\gamma^2+2\gamma)}]}+ \frac{C_4\exp(5aT)}{n}.
$$
Then, we have the following estimate for $a\geq 1$ 
$$
\begin{aligned}
\W_2^2\left(f_m(T),\t f_m(T)\right)\leq\& u_0 m a\exp(2aT)+\frac{C_2T\exp(2aT)m}{a\exp[C_1a^{4/(\gamma^2+2\gamma)}]}\\\leq\& m\left(u_0 +\frac{C_2T}{a^2\exp[C_1a^{4/(\gamma^2+2\gamma)}]}\right)a\exp(2aT)\\\leq\& m\left(u_0 +\frac{C_2T}{\exp[C_1a^{4/(\gamma^2+2\gamma)}]}\right)\exp(3aT),   
\end{aligned}
$$
where the constants $C_1$ and $C_2$ are independent of $a$, $m$, $n$ and $T$. We choose our suitable cut-off $a$ such that
$$
\exp\left(C_1a^{\frac{4}{\gamma^2+2\gamma}}\right)=\frac{u_0+1}{u_0},\quad\text{i.e.,}\quad a=\left(\frac{1}{C_1}\log\left(1+\frac{1}{u_0}\right)\right)^\frac{\gamma^2+2\gamma}{4}
$$
Then, 
$$
\begin{aligned}
\W_2^2\left(f_m(T),\t f_m(T)\right)\leq \& m \left(u_0 +\frac{C_2Tu_0}{u_0+1}\right)\exp\left[\frac{3T}{C_1^{(\gamma^2+2\gamma)/4}}\left(\log\left(1+\frac{1}{u_0}\right)\right)^\frac{\gamma^2+2\gamma}{4}\right]\\
\leq\& m u_0\left(1 +C_2T\right)\left[\exp\left(\log\left(1+\frac{1}{u_0}\right)\right)^\frac{\gamma^2+2\gamma}{4}\right]^{C_5(\gamma,T)}
\end{aligned}
$$
Since $\gamma^2+2\gamma<4$ for any $\gamma\in(0,1]$, it holds for any parameter $\eta'\in(0,\frac{1}{C_5})$ such that 
$$
\exp\left[\left(\log\left(1+\frac{1}{u_0}\right)\right)^\frac{\gamma^2+2\gamma}{4}\right]\leq \left(1+\frac{1}{u_0}\right)^{\eta'},
$$
which implies for any $\eta:=\eta'C_5\in(0,1)$
$$
\W_2^2\left(f_m(T),\t f_m(T)\right)\leq C(1+T)m u_0\left(1+\frac{1}{u_0}\right)^\eta\leq Cm(1+T) (u_0)^{1-\eta}.
$$
We conclude the stability estimate of weak solutions of the Landau hierarchy \eqref{eq Landau hierarchy} for each $m\in\N$.
    
\end{proof}

\section*{Acknowledgements}
The author thanks for the fruitful discussion with José A. Carrillo, Pierre-Emmanuel Jabin, Yifan Jiang and Dejun Luo.

\bigskip

\bibliographystyle{abbrv}
\bibliography{ref}

\end{document}